\newif\ifJ
\ifJ
\else
\documentclass{article}
\usepackage[a4paper]{geometry}
\usepackage[numbers]{natbib}
\usepackage{amsmath,amsfonts,amsthm,amssymb}
\usepackage{mathrsfs}
\usepackage{comment}
\usepackage{verbatim}
\usepackage[breaklinks,bookmarks=false]{hyperref}
\hypersetup{colorlinks, linkcolor=blue, citecolor=blue,
urlcolor=blue, plainpages=false, pdfwindowui=false,
pdfstartview={FitH}}
\usepackage{cleveref,color,enumitem,tensor,mathtools}
\usepackage{hyperref}
\usepackage[textsize=tiny,color=blue!10!white]{todonotes}
\usepackage{bm}

\usepackage[medium,compact]{titlesec}
\titleformat*{\section}{\large\bfseries}
\titleformat*{\subsection}{\bfseries}
\usepackage{cleveref}
\numberwithin{equation}{section}
\newtheorem{theorem}{Theorem}[section]{\bfseries}{\it}
\newtheorem{proposition}[theorem]{Proposition}{\bfseries}{\it}
\newtheorem{lemma}[theorem]{Lemma}{\bfseries}{\it}
\newtheorem{corollary}[theorem]{Corollary}{\bfseries}{\it}
{\bfseries}{\it}
\newtheorem{example}{Example}[section]{\bfseries}{\rmfamily}
\newtheorem{assumption}[theorem]{Assumption}{\bfseries}{\it}
\theoremstyle{definition}
\newtheorem{remark}{Remark}[section]{\bfseries}{\rmfamily}

\fi

\newcommand{\R}{\mathbb{R}}
\newcommand{\N}{\mathbb{N}}

\DeclareMathOperator{\supp}{supp}
\DeclareMathOperator{\diam}{diam}
\DeclareMathOperator{\argmax}{arg\,max}

\DeclareMathOperator{\Trace}{Tr}
\DeclareMathOperator{\Conv}{conv}

\newcommand{\abs}[1]{\lvert#1\rvert}
\newcommand{\norm}[1]{\left\lVert#1\right\rVert}
\newcommand{\pair}[2]{\langle#1,#2\rangle}
\newcommand{\Rdd}{\R^{d\times d}}
\newcommand{\Rddsym}{\R^{d\times d}_{\mathrm{sym}}}
\newcommand{\Rddsymp}{\R^{d\times d}_{\mathrm{sym},+}}
\newcommand{\tends}{\rightarrow}
\newcommand{\dx}{\mathrm{d}x}

\newcommand{\Cclass}{\mathcal{C}(\underline{\nu},\overline{\nu},\varepsilon)}
\newcommand{\dsym}{s_{d}}

\newcommand{\mbb}[1]{\mathbb{#1}}
\newcommand{\angled}[1]{\left \langle #1 \right\rangle}

\newcommand{\TITLESTRING}{Fully nonlinear second-order mean field games with nondifferentiable Hamiltonians}

\newcommand{\ABSTRACTSTRING}{We analyse fully nonlinear second-order mean field games (MFG) with nondifferentiable Hamiltonians, which take the form of a coupled system of a fully nonlinear Hamilton--Jacobi--Bellman equation and a Kolmogorov--Fokker--Planck partial differential inclusion (PDI) featuring the set-valued subdifferential of the Hamiltonian.
We show the existence of solutions of some stationary MFG systems with quite general coupling operators and nonnegative distributional source terms, on general bounded convex domains, under the primary assumptions of uniform ellipticity and the Cordes condition on the diffusion coefficient.
The existence proof is founded on an original, and equivalent, reformulation of the PDI as a nonstandard variational inequality (VI), that offers significant flexibility in passages to limits.
Furthermore, the uniqueness of the solution of the PDI/VI system is proved in the case of strictly monotone couplings.
We then show how the MFG PDI/VI system in the fully nonlinear setting can be obtained as the limit of a sequence of PDE systems with differentiable Hamiltonians, and we give further results on the continuous dependence of the solution.\medskip 

\noindent\textbf{Keywords.} Mean field games, fully nonlinear, partial differential inclusions, variational inequalities, existence and uniqueness of solutions.\medskip

\noindent\textbf{Mathematics Subject Classification.} 35J57, 35Q89, 49N80, 49J53.
}
\begin{document}

\ifJ

\else

\title{\TITLESTRING}
\author{Thomas Sales\footnotemark[1]~~and Iain Smears\footnotemark[1]}
\date{}
\maketitle

\renewcommand{\thefootnote}{\fnsymbol{footnote}}

\footnotetext[1]{Department of Mathematics, University College London, Gower Street, WC1E 6BT London, United Kingdom (\texttt{t.sales@ucl.ac.uk}, \texttt{i.smears@ucl.ac.uk})}

\begin{abstract}
\ABSTRACTSTRING
\end{abstract}

\fi

\section{Introduction}\label{section: introduction}
 
\subsection{Fully nonlinear mean field games}
Mean field games (MFG), introduced by Lasry \& Lions~\cite{LasryLions2006i,LasryLions2006ii,LasryLions2007} and, independently, by Huang, Malham\'e \& Caines \cite{HuangMalhameCaines2006}, are models of the Nash equilibria of dynamic differential games for large populations of players.
In these works, the Nash equilibria can be described by a system of coupled nonlinear partial differential equations (PDE), including a Hamilton--Jacobi--Bellman equation (HJB), originating from the underlying optimal control problem of a typical player, and a Kolmogorov--Fokker--Planck (KFP) equation, describing the evolution of the density of players under the players' controlled stochastic dynamics.
In models where the players' controls only enter the deterministic part of the stochastic dynamics, with nondegenerate noise, the resulting MFG system consists of second-order quasilinear PDE.
We refer the reader to~\cite{AchdouCardaliaguetDelaruePorettaSantambrogio2020} for an introduction to the quasilinear setting.
However, in the more general case where the controls can also enter the diffusion terms, the resulting PDE becomes fully nonlinear.

Consider, as model problems, fully nonlinear second-order elliptic MFG systems of the form
\begin{subequations}\label{eqn:mfg_pde_system}
\begin{alignat}{2}
H(x,D^2 u) &= F[m](x) &\quad&\text{in }\Omega, \label{eqn: stationary MFG1}
 \\ D^2:\left(m\frac{\partial H}{\partial M}(x,D^2 u)\right)&= G(x) &\quad&\text{in }\Omega, \label{eqn: stationary MFG2}
\\
u=0, \quad m&=0 &\quad&\text{on }\partial \Omega,
\end{alignat}
\end{subequations}
where $\Omega\subset \R^d$, $d\geq 1$, denotes the state space of the game, where the function $u\colon \Omega\to \R$ denotes the value function, and where $m\colon \Omega\to \R$  denotes the density of players.
In the analysis below, we will assume that the domain~$\Omega$ is nonempty, open, bounded and convex.
This will imply that the boundary~$\partial\Omega$ will be at least Lipschitz regular~\cite[Corollary~1.2.2.3]{Grisvard2011}, but we will not make any further smoothness assumptions on~$\partial\Omega$.
The convexity of~$\Omega$ will be used below as one of the ingredients to obtain $H^2$ regularity of the solution of the HJB equation; we refer the reader to Remark~\ref{rem:convexity_omega} below for a more detailed discussion of the role of this assumption.
The Hamiltonian~$H\colon \overline{\Omega}\times \Rddsym\to \R$, which appears above, is defined in terms of the underlying stochastic optimal control problem by
\begin{equation}\label{eqn:H_def}
    H(x,M)\coloneqq \sup_{\alpha\in\mathcal{A}}\left\{-a(x,\alpha):M - f(x,\alpha) \right\}
    \quad \forall (x,M)\in \overline{\Omega}\times\Rddsym,
\end{equation}
where $\mathcal{A}$ denotes the set of player controls, where $a\colon \overline{\Omega}\times \mathcal{A}\to \Rddsymp$ is the  controlled diffusion tensor, and where the function $f\colon \overline{\Omega}\times \mathcal{A}\to \R$ is the control-dependent component of the players' cost. 
Here, $\Rddsym$ denotes the space of real symmetric $d\times d$ matrices and~$\Rddsymp$ is the subset of matrices in~$\Rddsym$ that are additionally positive semidefinite.
To explain the notation in~\eqref{eqn: stationary MFG2}, the partial derivative $\frac{\partial H}{\partial M}$ refers to the matrix-valued partial derivative of $H$ with respect to its second argument, with values in $\Rddsym$, and $D^2:$ denotes the double divergence operator, i.e.\ $D^2:W\coloneqq\sum_{i,j=1}^d\partial_{x_i x_j}W_{ij}$ for all sufficiently regular matrix-valued functions $W\colon \Omega\to\Rddsym$. 
Note that in order to write~\eqref{eqn: stationary MFG2}, it is very often assumed that $H$ is differentiable in $M$.
However, we shall avoid differentiability assumptions and instead allow~$H$ to be nondifferentiable in $M$, with the only assumptions on the data in~\eqref{eqn:H_def} being that $\mathcal{A}$ is a given compact metric space, and that the functions $a$ and $f$ are continuous over $\overline{\Omega}\times\mathcal{A}$.
We will explain in much more detail below how the system~\eqref{eqn:mfg_pde_system} should be understood when $H$ is not required to be differentiable in $M$.
The coupling term~$F$ is the part of the players' cost functional that depends on the overall density of players; in the analysis we will allow it to be a possibly nonlinear, nonlocal operator mapping~$L^2(\Omega)$ into itself.
Regarding the boundary condition,  we consider for simplicity homogeneous Dirichlet boundary conditions for $u$ and $m$ in the system~\eqref{eqn:mfg_pde_system} above.
Dirichlet boundary conditions arise when the underlying stochastic optimal control problem of the players involves stopping at the first-exit time from $\Omega$, i.e.\ players exit the game upon reaching the boundary~$\partial \Omega$. 
The source term $G$ appearing on the right-hand side of~\eqref{eqn: stationary MFG2} represents the source of new players entering the game.
Naturally, the problem~\eqref{eqn:mfg_pde_system} can be seen as a simplified steady-state system motivated by a corresponding time-dependent problem.

Mean field game models with controlled volatility have been considered in the stochastic analysis literature by Lacker~\cite{Lacker2015}, who has shown the existence of Nash equilibria formulated in terms of the underlying stochastic optimal control problem.
At present, there are however only a handful of papers on fully nonlinear MFG from the PDE analysis perspective.
Andrade and Pimentel~\cite{AndradePimentel2021} have studied some particular fully nonlinear MFG systems of potential type, where the system can be derived from the minimization of a global functional for the Hessian of the value function.
In~\cite{ChowdhuryJakobsenKrupski2024,ChowdhuryJakobsenKrupski2025}, Chowdhury, Jakobsen \& Krupski have analysed time-dependent problems where the Hamiltonian takes the form $H(\mathcal L u)$, where $H\in C^1(\R)$ and $\mathcal{L}$ is a possibly nonlocal diffusion operator tied to a L\'{e}vy process appearing in the stochastic dynamics; in the case of Brownian motion without deterministic drift, this reduces to a Hamiltonian of the form $H(\Delta u)$, see~\cite[Footnote~3, p.~6306]{ChowdhuryJakobsenKrupski2024}.
They have shown the existence of a classical--very weak solution for the time-dependent system, for problems without boundary conditions, under suitable assumptions on the data.
In another recent work, Ignazio \& Ricciardi~\cite{IgnazioRicciardi2025} also consider problems on the whole space $\R^d$, also in the setting of classical--very weak solutions where the fully nonlinear part of the Hamiltonian also only depends on the Laplacian~$\Delta u$, rather than the full Hessian~$D^2 u$.
Ferreira and Gomes have also outlined in~\cite[Section~7]{FerreiraGomes2018} some possible approaches to showing existence of Minty weak solutions of fully nonlinear MFG if some global monotonicity structure is available, e.g. if the coupling operator $F$ above is assumed to be monotone.

The analysis of the PDE system~\eqref{eqn:mfg_pde_system} presents several challenges that motivate a different approach to prior works.
First, the system does not generally represent the optimality conditions for optimizers of some potential.
The coupling operator is also not required to be monotone for the existence theory below.
Second, the Hamiltonian $H$, defined in~\eqref{eqn:H_def}, is allowed to depend on the full Hessian $D^2 u$, rather than only on its Laplacian $\Delta u$.
Third, the lack of smoothness of the boundary $\partial \Omega$ generally limits the possible regularity (up to the boundary) of the solution.
In addition to the above, the most significant challenge that we address here is that, in many applications, the Hamiltonian $H$ is generally not differentiable with respect to the Hessian variable~$M$.
As a very simple example, consider $\mathcal{A}=[0,1]$ and $a(x,\alpha)=(1+\alpha)I_d$, with $I_d$ the $d\times d$ identity matrix, and $f\equiv 0$, in which case $H(x,M)=\max\{-\Trace M,-2\Trace M\}$, which is clearly not differentiable for all $M$.
If $H$ is nondifferentiable, then it is clear that the term involving $\frac{\partial H}{\partial M}$ in the KFP equation~\eqref{eqn: stationary MFG2} might not exist as a well-defined function for general $u$, since the Hessian $D^2 u$ might take values at points of nondifferentiability of $H$ on a set of positive measure in~$\Omega$. 
We stress that Lipschitz continuity of~$H$ is not sufficient to avoid this difficulty.
Therefore, when $H$ is nondifferentiable, the system~\eqref{eqn:mfg_pde_system} must be understood in a suitably generalized sense, as we detail further below.

\subsection{MFG with nondifferentiable Hamiltonians and main results}

The analysis of MFG systems with nondifferentiable Hamiltonians has recently been developed in the sequence of papers~\cite{OsborneSmears2024i,OsborneSmears2024ii,OsborneSmears2025i,OsborneSmears2025ii,OsborneSmears2025iv,OsborneSmears2025iii} in the context of second-order quasilinear MFG systems.
It is well-known in optimal control that nondifferentiability of the Hamiltonian is related to the nonuniqueness of the optimal feedback controls for the underlying stochastic optimal control problem of the players.
The insight is that, in general, players occupying the same state may have to use different optimal feedback controls in order to sustain a Nash equilibrium. 
The question of determining the players' choices of controls then becomes implicitly part of the problem to be solved.
Mathematically, this is expressed by the generalization of the usual KFP equation as a \emph{partial differential inclusion} (PDI) where the partial derivative of the Hamiltonian is replaced by its set-valued subdifferential.
We refer the reader in particular to~\cite{OsborneSmears2025iii} for both a heuristic derivation in terms of the underlying stochastic optimal control problems, and also a more rigorous analysis in the quasilinear setting showing that the resulting MFG PDI system is, in some sense, the limit of classical MFG PDE systems under regularization of the Hamiltonian.

In the context of the present problem~\eqref{eqn:mfg_pde_system}, the appropriate generalization of the problem for nondifferentiable $H$ is in terms of the PDI system: find a pair of functions $(u,m)$, with appropriate regularity, such that
\begin{subequations}\label{eqn:mfg_pdi_system}
\begin{alignat}{2}
H(x,D^2 u) &= F[m](x) &\quad&\text{in }\Omega, \label{eqn:mfg_pdi_1}
\\ D^2:\left(-\overline{a} m\right) &= G(x) &\quad&\text{in }\Omega, \label{eqn:mfg_pdi_2}
\\
u=0, \quad m&=0 &\quad&\text{on }\partial \Omega,
\end{alignat}
for some measurable function $\overline{a}\colon \Omega\to \Rddsymp$ satisfying 
\begin{equation}\label{eq:mfg_pdi_system_inclusion}
\begin{aligned}
-\overline{a}(x)\in \partial_M H(x,D^2 u (x)) &&& \text{for a.e. } x\in\Omega,
\end{aligned}
\end{equation}
\end{subequations}
where $\partial_M H$ denotes the subdifferential of $H$.
We will make the functional setting and problem statement more precise in Section~\ref{sec:notation} below, in particular see~\eqref{eq:MFG_PDI} for the precise formulation of~\eqref{eqn:mfg_pdi_system}.
The equation~\eqref{eqn:mfg_pdi_2} and inclusion~\eqref{eq:mfg_pdi_system_inclusion} above can be jointly expressed, at least formally, as 
\begin{equation*}
\begin{aligned}
G \in D^2:\left( m \partial_M H(x,D^2 u) \right) &&&\text{in }\Omega.
\end{aligned}
\end{equation*}
Observe that the matter of determining a suitable $\overline{a}$ satisfying the inclusion condition~\eqref{eq:mfg_pdi_system_inclusion} is implicitly part of the problem to be solved, and such an $\overline{a}$ is not necessarily unique; see~\cite{OsborneSmears2025iii} for some examples in the quasilinear setting.
It is also known from examples in the quasilinear setting that discontinuities in the coefficients appearing in the inclusion can lead to loss of interior regularity of the solution, see e.g.~\cite[Example~2]{OsborneSmears2024i}.
Note also that the PDI problem~\eqref{eqn:mfg_pdi_system} simplifies to the usual PDE problem~\eqref{eqn:mfg_pde_system} if $H$ is differentiable everywhere in~$M$.

Our first main result, Theorem~\ref{thm:PDI_VI_equivalence}, shows that a pair $(u,m)$ is a solution of the partial differential inclusion problem if, and only if, it is a solution of the variational inequality problem, as defined in Section~\ref{subsection:problem_formulations}.
Next, in Theorem~\ref{thm: MFG existence} below, we show the existence of solutions of the PDI system~\eqref{eqn:mfg_pdi_system}, under appropriate assumptions on the problem data.
More precisely, we will work with strong--very weak solutions with value function $u\in H^2(\Omega)\cap H^1_0(\Omega)$ and density $m\in L^2(\Omega)$.
The precise assumptions on the problem data are given in Section~\ref{sec:notation} below.
In Theorem~\ref{thm:MFG_uniqueness} below, we also show the uniqueness of the solution if the coupling~$F$ is additionally strictly monotone, thus extending the well-known uniqueness result of Lasry \& Lions. 
We also show in Theorem~\ref{thm: regularization convergence} that the MFG PDI problem can be obtained as the limit of MFG PDE with regularized Hamiltonians, thus extending results of~\cite{OsborneSmears2025iii} to the fully nonlinear setting.

\subsection{Approach to the analysis and further contributions}
Let us now explain the motivation for several key features of our approach, and how these address the main challenges appearing in the analysis.
First, it is clear that the nondifferentiability of~$H$, and the resulting lack of good continuity properties of the set-valued subdifferential maps, are major obstacles in the analysis.
For example, in a typical situation encountered below, we may have a sequence of functions $u_j \in H^2(\Omega)\cap H^1_0(\Omega)$ converging to some limit $u$ in the norm of $H^2(\Omega)$, and also a sequence of functions $m_j\in L^2(\Omega)$ that converges only weakly to some limit $m$ in $L^2(\Omega)$, and we would like to pass to limits in the various terms appearing in~\eqref{eqn:mfg_pdi_system}.
A critical difficulty appears when considering some sequence of selections $\overline{a}_j\in L^\infty(\Omega;\Rddsymp)$ such that $-\overline{a}_j(x)\in\partial_M H(x,D^2u_j(x))$ for a.e.\ $x\in \Omega$, since it is easy to find examples where the lack of differentiability of $H$ causes the sequence $\{\overline{a}_j\}_j$ to be weakly-$*$ converging along some subsequences, but there is no subsequence that is converging strongly, see for instance the example in Remark~\ref{rem:no_strong_continuity} below.
Given only weak convergence, it is difficult to establish any convergence properties of the sequence of products~$\overline{a}_jm_j$ that appear in the principal term of the KFP equation. 
Note that this problem is more critical in the fully nonlinear setting than in the quasilinear one treated in~\cite{OsborneSmears2024i,OsborneSmears2025i,OsborneSmears2025iv,OsborneSmears2025iii}, as in the latter case there is some additional regularity for the densities allowing for some strong compactness by the Rellich--Kondrachov theorem.
Thus the approaches used in existence proofs of earlier works do not extend straightforwardly to the fully nonlinear setting.

In order to address this difficulty, the first main ingredient of our approach is a reformulation of the PDI system, in particular the KFP inclusion, as a nonstandard variational inequality~(VI) problem, see Section~\ref{sec:notation} below, in particular \eqref{eq:MFG_VI}. 
The VI problem takes a rather original form, as the subdifferential $\partial_M H$ no longer makes an explicit appearance.
We emphasize that the VI formulation of the problem considered here differs substantially from the variational inequality problem in~\cite{FerreiraGomes2018}, which requires differentiability of the Hamiltonian; see Remark~\ref{rem:other_vi} below for a more detailed explanation of the essential differences.
One of the central contributions of this work, stated in Theorem~\ref{thm:PDI_VI_equivalence}, and its Corollary~\ref{cor:PDI_VI_equivalence_complete}, is that the PDI problem~\eqref{eq:MFG_PDI} and the VI problem~\eqref{eq:MFG_VI} are equivalent under quite weak conditions on the problem data. 
The proof of Theorem~\ref{thm:PDI_VI_equivalence} is essentially founded on the verification of a minimax principle for some functional defined on the test functions and the measurable selections of the subdifferential, see Section~\ref{section: equivalence proof} below.
The principal benefit of working with the VI formulation of the problem is that it greatly simplifies passing to limits in the various terms of the problem, thereby avoiding the difficulties described in the previous paragraph.
Note, however, that we do not work exclusively with the VI formulation, as certain key steps will make use also of the PDI form, so both points of view are complementary.
This forms the cornerstone of the proof of Theorem~\ref{thm: MFG existence} on the existence of solutions, as it enables the verification of the hypotheses of Kakutani's fixed point theorem for some suitably chosen set-valued map.
The VI formulation of the problem is also useful for handling various limiting arguments in Section~\ref{section: regularization} where we consider the effects of regularization and perturbation of the data.
More generally, it will become clear in the analysis below that the equivalence of PDI and VI formulations can be generalized to many other MFG systems, and not just the one considered here, see~Remark~\ref{remark: minimax theorems} below, so we expect that it may be a useful tool for a much wider range of problems.

The second key ingredient of our approach is motivated by another difficulty that appears for fully nonlinear MFG systems with nondifferentiable Hamiltonians.
The problem is that, in the KFP equation~\eqref{eqn:mfg_pdi_2}, the coefficient~$\overline{a}\in L^\infty(\Omega;\Rddsymp)$ will generally be discontinuous, regardless of the regularity of the value function~$u$.
Thus the KFP equation involves the formal adjoint of a second-order elliptic operator in nondivergence form with a discontinuous coefficient. 
It is well-known that the well-posedness of such equations is quite delicate: see~\cite[p.~18]{MaugeriPalagachevSoftova2000} for an example of nonuniqueness of strong solutions, and see also~\cite{Nadirashvili1997,Safonov1999}  for nonuniqueness of extended notions of solution, including viscosity solutions (with measurable ingredients).
We emphasize that these issues arise from the combination of discontinuity of the coefficient and the nondivergence structure of the second-order terms, in marked contrast to second-order elliptic equations in divergence form.
However, well-posedness of linear nondivergence form elliptic PDE with discontinuous coefficients (and their adjoints) can be recovered if the leading coefficients satisfy a Cordes condition~\cite{Cordes1956,MaugeriPalagachevSoftova2000}.
This motivates our assumption that the diffusion coefficients appearing in the Hamiltonian satisfy the Cordes condition, see~\eqref{eq:Cordes_condition} below for a precise statement. 
Note that in the case of problems in one or two space dimensions, i.e.\ $d\leq 2$, uniform ellipticity implies the Cordes condition~\eqref{eq:Cordes_condition}, so the condition is only more restrictive than uniform ellipticity for $d\geq 3$.
As further motivation, it was shown by S\"uli and the second author in~\cite{SmearsSuli2013,SmearsSuli2014,SmearsSuli2016} that the Cordes condition also allows for an effective analysis of strong solutions of fully nonlinear HJB equations on nonsmooth convex domains, with good quantitative stability bounds in $H^2(\Omega)\cap H^1_0(\Omega)$.
Problems with Cordes coefficients have since been studied in other contexts, e.g.\ fully nonlinear Isaacs equations~\cite{KaweckiSmears2021,KaweckiSmears2022}, Monge--Amp\`ere equations~\cite{GallistlTran2023,GallistlTran2024}, and homogenization theory~\cite{SprekelerSuliCapdeboscq2020,GallistlSprekelerSuli2021,SprekelerSuliZhang2025}.
We briefly detail several key well-posedness results for the HJB and KFP equations with Cordes coefficients, when considered separately, in Section~\ref{sec:HJB_KFP_cordes} below.
We will also show below, in Theorem~\ref{thm:comparison_nondivergence} and its Corollary~\ref{cor:kfp_comparison}, that a comparison principle is available for operators with Cordes coefficients, which requires less integrability on the second-order derivatives of functions than the usual Aleksandrov--Bakelman--Pucci maximum principle when $d\geq 3$; this comparison principle extends some ideas in~\cite{SprekelerSuliZhang2025}, which treated periodic boundary conditions on a torus, to Dirichlet boundary-value problems on nonsmooth convex domains.
The comparison principle shown here is, to the best of our knowledge, also another original contribution.
These results provide fundamental tools for the analysis of the MFG system.

The outline for this paper is as follows. Section~\ref{sec:notation} below sets out the notation and assumptions used for the analysis, and states the PDI and VI formulations of the problem.
The main results are gathered in Section~\ref{sec:main_results}, including the theorems on the equivalence of the PDI and VI problems, as well as the existence and uniqueness of solutions.
Section~\ref{section: equivalence proof} contains the proof of the equivalence of the PDI and VI formulations.
Then, in Section~\ref{sec:HJB_KFP_cordes}, we summarize some key results on the HJB and KFP equations with Cordes coefficients on convex domains, alongside the comparison principle for problems with Cordes coefficients.
The proofs of the existence and uniqueness results are given in Section~\ref{section: well-posed}.
In Section~\ref{section: regularization}, we consider the properties of the MFG system under regularization of the Hamiltonian and perturbations of the data.
Appendix~\ref{appendix: comparison principle} contains the proof of the comparison principle.
Appendix~\ref{sec:Morea-Yosida_representation} gives a representation formula for Moreau--Yosida regularizations of convex functions, that justifies the application of the theory of Section~\ref{section: regularization} to this family of regularizations of the Hamiltonian.

\section{Notation and formulation of the problem}\label{sec:notation}

Recall that $\Omega$ is assumed to be a nonempty, open, bounded, convex subset of $\R^d$, $d\geq 1$.
Let $L^2(\Omega)$ denote the usual space of Lebesgue measurable square-integrable extended real-valued functions on $\Omega$, and let $(\cdot,\cdot)_\Omega$ denote the inner-product on $L^2(\Omega)$.
Let $L^2_+(\Omega)$ denote the subset of functions in $L^2(\Omega)$ that are nonnegative a.e.\ in $\Omega$.
For each integer $k\in\N$, let $H^k(\Omega)\coloneqq W^{k,2}(\Omega)$ and $H^k_0(\Omega)\coloneqq W^{k,2}_0(\Omega)$ denote the usual Sobolev spaces, see~\cite{AdamsFournier03}, equipped with their usual norms and inner products. 
As we will be interested in strong solutions of the HJB equation~\eqref{eqn: stationary MFG1} satisfying a homogeneous Dirichlet boundary condition on $\partial \Omega$, we will work with the space $V$ defined by
\begin{equation}\label{eq:V_def}
V \coloneqq  H^2(\Omega) \cap H^1_0(\Omega).
\end{equation}
The space~$V$ is a Hilbert space equipped with the same norm and inner-product as $H^2(\Omega)$.
The source term $G$ appearing in~\eqref{eqn:mfg_pdi_2} is allowed to be distributional.
More precisely, we suppose that~$G\in V^*$, where $V^*$ denotes the dual space of~$V$.
We will also frequently assume that~$G$ is nonnegative, by which we mean that $\angled{G, v}_{V^*\times V} \geq 0$ for all nonnegative a.e.\ functions $v \in V$, where $\angled{\cdot ,\cdot}_{V^*\times V}$ denotes the duality pairing between~$V$ and~$V^*$.
This allows for many interesting examples, such as sufficiently regular nonnegative measures, including Dirac measures when $d\leq 3$.

We assume that $F\colon L^2(\Omega)\to L^2(\Omega)$ is a given operator.
To highlight the fact that $F$ is possibly nonlinear and nonlocal, we write $F\colon m\mapsto F[m]$ for each $m\in L^2(\Omega)$.
In the analysis of the existence of solutions developed below, we will also assume that $F$ is \emph{completely continuous}, i.e.\ if $\{m_j\}_{j\in\N}$ is a weakly converging sequence in $L^2(\Omega)$ with weak  limit $m$, then $F[m_j] \rightarrow F[m]$ strongly in $L^2(\Omega)$ as $j\to \infty$.
Note that such completely continuous operators are bounded, i.e.\ they map bounded sets to bounded sets.
We do not require any growth condition on~$F$.
As a remark, it is well-known that if the operator~$F$ is linear, then~$F$ is completely continuous if, and only if, it is compact, see~\cite[Chapter VI]{Conway1990}.
In the later analysis of the uniqueness of the solution, we will additionally require that the operator~$F$ satisfies a strict monotonicity condition, as is standard in the mean field game literature~\cite{LasryLions2006i,LasryLions2006ii,LasryLions2007}.
\begin{assumption}
\label{assumption: strict monotonicity}
The operator $F\colon L^2(\Omega)\to L^2(\Omega)$ is \emph{strictly monotone} on $L^2_+(\Omega)$, i.e.\ if $\left( F[m_1] - F[m_2], m_1 - m_2\right)_\Omega \leq 0 $ for some $m_1, m_2 \in L^2_+(\Omega)$, then $m_1=m_2$.
\end{assumption}

We emphasize that Assumption~\ref{assumption: strict monotonicity} only requires strict monotonicity of~$F$ on~$L^2_+(\Omega)$, the subset of nonnegative functions in~$L^2(\Omega)$.
Note also that this assumption is only used for the uniqueness of solutions, but is not required for the equivalence of notions of solution nor for their existence.
 
\begin{example}[Examples of coupling operators]
\label{example:strictly_monotone_operator}
Simple examples of operators $F\colon L^2(\Omega)\tends L^2(\Omega)$ that are completely continuous are provided by Hilbert--Schmidt operators of the form $F_1[m](x)=\int_
\Omega K(x,y)m(y)\,\mathrm{d}y$ for all $x\in \Omega$, where $K\in L^2(\Omega\times \Omega)$, see \cite[Lemma 8.20, Theorems~8.83 \& 8.84]{RenardyRogers2004}.
Furthermore, if the kernel is of the form $K(x,y) = \sum_{n=1}^\infty \mu_n \varphi_n(x) \varphi_n(y)$ where $\{\mu_n\}_{n \in \N}\in \ell^2$ is a square-summable sequence of strictly positive numbers, and $\{\varphi_n\}_{n\in \N}$ is an orthonormal basis of~$L^2(\Omega)$, then $K\in L^2(\Omega\times\Omega)$ and the operator~$F_1$ defined above is additionally strictly monotone on $L^2(\Omega)$. 
Indeed, it is easy to check that 
\[
(F_1[m_1]-F_1[m_2],m_1-m_2)_\Omega=\sum_{n=1}^{\infty}\mu_n|(\varphi_n,m_1-m_2)_\Omega|^2 \quad \forall m_1,\,m_2 \in L^2(\Omega).
\]
Therefore, if $(F_1[m_1]-F_1[m_2],m_1-m_2)_\Omega\leq 0$ then $(m_1-m_2,\varphi_n)_\Omega=0$ for all $n\in\N$, and hence $m_1=m_2$ since $\{\varphi_n\}_{n\in\N}$ forms an orthonormal basis of $L^2(\Omega)$.
In fact, the Hilbert--Schmidt Theorem, see~\cite[Theorem 8.94]{RenardyRogers2004}, shows that all strictly monotone, completely continuous, self-adjoint integral operators from $L^2(\Omega)$ into itself are of this form.

One can also find examples of operators that are strictly monotone on~$L^2_+(\Omega)$ but not all of~$L^2(\Omega)$.
For example, consider $F_2[m] \coloneqq  F_1[m] + \left(\int_{\Omega} m(x) \, \dx \right)^2$ for all $m\in L^2(\Omega)$, with strictly monotone $F_1$ defined above.
It is clear that $F_2$ is completely continuous from $L^2(\Omega)$ into itself, and obeys the growth condition~$\norm{F_2[m]}_{L^2(\Omega)} \leq C(1+\norm{m}_{L^2(\Omega)}^2)$ for some constant~$C$.
Furthermore, the operator~$F_2$ is strictly monotone on $L^2_+(\Omega)$ because $F_1$ is strictly monotone on $L^2(\Omega)$ and because $m\mapsto \left(\int_{\Omega} m(x) \, \dx \right)^2$ is monotone on $L^2_+(\Omega)$ owing to the identity $(s^2-t^2)(s-t)=(s+t)(s-t)^2\geq 0$ for all nonnegative real numbers $s$ and $t$. 
However, $F_2$ is not strictly monotone on all of $L^2(\Omega)$, which is seen by taking $m_1\equiv -1$ and $m_t\equiv - t$, $t>0$, and using the linear growth of $F_1$ to find that 
\[
(F_2[m_t]-F_2[m_1],m_t-m_1)_\Omega=-|\Omega|^3(t+1)(t-1)^2+\mathcal{O}(t^2+1)\tends -\infty,
\]
as $t\tends \infty$, where $|\Omega|>0$ denotes the Lebesgue measure of $\Omega$.
\end{example}

\subsection{The Hamiltonian and its subdifferential.}
Recall that the Hamiltonian $H$ is defined in~\eqref{eqn:H_def} above, and recall that the control set $\mathcal{A}$ is assumed to be a compact metric space, and the functions $a\colon \overline{\Omega}\times \mathcal{A}\to \Rddsymp$ and $f\colon \overline{\Omega}\times\mathcal{A}\to \R$ are assumed to be continuous.
It follows that $H$ is real-valued, and also convex with respect to its second argument.
Furthermore, the hypotheses above also imply that $H$ is Lipschitz continuous with respect to $M$, in particular
\begin{equation}\label{eq:H_Lipschitz}
\begin{aligned}
\abs{H(x,M_1)-H(x,M_2)}\leq L_H|M_1-M_2| \quad\forall M_1,\,M_2 \in \Rddsym, 
\end{aligned}
\end{equation}
where, in a slight abuse of notation, $\abs{\,\cdot\,}$ denotes the Frobenius norm on $\Rdd$, and where the constant $L_H$ is defined by
\begin{equation}\label{eq:L_H_def}
L_H \coloneqq \sup_{(x,\alpha)\in \overline{\Omega}\times\mathcal{A}}|a(x,\alpha)| <\infty.
\end{equation}
The continuity of the functions $a$ and $f$ on $\overline{\Omega}\times\mathcal{A}$ and the Lipschitz continuity of~$H$, see~\eqref{eq:H_Lipschitz} above, imply that the operator $H[\,\cdot\,]\colon V\to L^2(\Omega)$ defined by
\begin{equation}\label{eq:H_op_def}
\begin{aligned}
H[v](x) \coloneqq H(x,D^2 v (x)) &&& \text{a.e. }x\in \Omega, \quad\forall v\in V,
\end{aligned}
\end{equation}
is well-defined, and is moreover also Lipschitz continuous.

Since $H$ is real-valued, convex, and continuous in its second argument, its partial subdifferential $\partial_M H(x,M)$, defined by
\begin{equation}\label{eq:subdifferential_def}
    \partial_M H(x,M) \coloneqq \{ A\in \R^{d\times d}_{\mathrm{sym}} \colon H(x,M)+A:(\widetilde{M}-M) \leq H(x,\widetilde{M}) \quad \forall \widetilde{M}\in \R^{d\times d}_{\mathrm{sym}}
      \},
\end{equation}
is a nonempty, closed and convex subset of $\Rddsym$ for each $(x,M)\in \overline{\Omega}\times \Rddsym$. Moreover, by Lipschitz continuity of~$H$, the subdifferential $\partial_M H(x,M)\subset \overline{B_{\Rddsym}(0,L_H)}$, the closed ball of radius $L_H$ in $\Rddsym$ w.r.t.\ the Frobenius norm. 
The following characterization of the subdifferential sets is useful in order to deduce their properties in terms of the properties of the diffusion coefficient function $a$.
\begin{proposition}\label{prop:subdifferential_characterization}
Let $\mathcal{A}$ be a compact metric space, and let $a\in C(\overline{\Omega}\times\mathcal{A};\Rddsymp)$ and $f\in C(\overline{\Omega}\times\mathcal{A})$. Let $H$ be defined by~\eqref{eqn:H_def}.
For each $M\in \Rddsym$ and $x\in\overline{\Omega}$, let $\Lambda(x,M)\subset \mathcal{A}$ be defined by
\begin{equation}\label{eq:Lambda_def}
\Lambda(x,M) \coloneqq \argmax_{\alpha \in \mathcal{A}}\left\{ -a(x,\alpha):M - f(x,\alpha)  \right\}.
\end{equation} 
The set $\Lambda(x,M)$ is nonempty and compact in $\mathcal{A}$.
For each $(x,M)\in \overline{\Omega}\times \Rddsym$, we have
\begin{equation}\label{eq:subdifferential_characterization}
\partial_M H(x,M) = \Conv \left\{ - a(x,\alpha): \alpha \in \Lambda(x,M)  \right\}.
\end{equation}
\end{proposition}
Note that the set $\Lambda(x,M)$ denotes the set of optimal feedback controls associated to $(x,M)$.
Thus, Proposition~\ref{prop:subdifferential_characterization} shows that the subdifferential of the Hamiltonian is the convex hull of the set of opposite diffusion tensors corresponding to the optimal feedback controls.
\begin{proof}
The result is perhaps well-known, so we include an elementary proof only for the sake of completeness.
It can be shown that the compactness of~$\mathcal{A}$ and the continuity of the functions $a$ and $f$ imply that the set-valued map $\Lambda\colon(x,M)\rightrightarrows \Lambda(x,M)$ is upper semicontinuous, and that $\Lambda(x,M)$ is nonempty and closed, hence compact, in $\mathcal{A}$ for each $(x,M)\in \overline{\Omega}\times \Rddsym$, see for instance~\cite[Lemma~9]{SmearsSuli2014}.
We now show~\eqref{eq:subdifferential_characterization}. Let $(x,M)\in \overline{\Omega}\times\Rddsym$ be fixed but arbitrary.
By continuity of the data and compactness of $\Lambda(x,M)$, the set $\left\{ - a(x,\alpha): \alpha \in \Lambda(x,M)  \right\}$ is also compact, and hence the convex hull $\Conv\left\{ - a(x,\alpha): \alpha \in \Lambda(x,M)  \right\}$ is compact in $\Rddsym$.
It is straightforward to check that $\Conv \left\{ - a(x,\alpha): \alpha \in \Lambda(x,M)  \right\}$ is a subset of $\partial_M H(x,M)$.
To show the converse and complete the proof, it is enough to show that the complement of $\Conv \left\{ - a(x,\alpha): \alpha \in \Lambda(x,M)  \right\}$ is a subset of the complement of $\partial_M H(x,M)$.
To this end, let $A\in \Rddsym\setminus \Conv \left\{ - a(x,\alpha): \alpha \in \Lambda(x,M)  \right\}$ be arbitrary. 
Then, by the Hyperplane Separation Theorem, there exists a $W\in\Rddsym$, with $|W|=1$, and a $\rho>0$ such that $-a(x,\alpha):W \leq A:W-\rho $ for all $\alpha\in\Lambda(x,M)$. 
Since $a$ is continuous, there exists a neighbourhood $U$ of $\Lambda(x,M)$ such that $-a(x,\alpha):W \leq A:W-\rho/2$ for all $\alpha\in U$. 
Then, by upper semicontinuity of~$\Lambda$, there is a $\delta>0$ such that $\Lambda(x,\widetilde{M})\subset U$ for all $\widetilde{M}\in \Rddsym$ such that $\abs{\widetilde{M}-M}\leq\delta$. 
It is then easy to check that $H(x,M+\delta W)\leq H(x,M)+A:(\delta W)-\delta\rho/2$, which shows that $A\not\in\partial_M H(x,M)$. 
This shows~\eqref{eq:subdifferential_characterization} and completes the proof.
\end{proof}

We now define the set-valued map $\mathcal{D}H\colon V \rightrightarrows L^\infty(\Omega;\Rddsym) $ by
\begin{equation}\label{eq:DH_def}
\mathcal{D}H[v] \coloneqq \left\{ A \in L^\infty(\Omega;\R^{d\times d}_{\mathrm{sym}}) : A(x) \in \partial_M H(x,D^2 v(x)) \quad \text{for a.e. } x\in \Omega \right\},
\end{equation}
for each $v\in V$.
In other words, $\mathcal{D}H[v]$ is the set of all measurable selections of the set-valued map $x\mapsto \partial_M H(x,D^2 v(x))$.
The following Lemma shows that $\mathcal{D}H[v]$ is a nonempty bounded convex subset of $L^\infty(\Omega;\Rddsym)$ for all $v\in V$.
\begin{lemma}\label{lem:DH_nonempty_bounded}
For each $v\in V$, the set $\mathcal{D}H[v]$ is nonempty and convex, and satisfies the uniform boundedness property
\begin{equation}\label{eq:DH_uniform_boundedness}
\begin{aligned}
\mathcal{D}H[v] \subset \{ A\in L^\infty(\Omega;\Rddsym) : \norm{A}_{L^\infty(\Omega;\Rdd)}\leq L_H \}  &&& \forall v\in V.
\end{aligned}
\end{equation}
\end{lemma}
\begin{proof}
Consider a fixed but arbitrary $v\in V$.
Observe that if a measurable function $A\colon \Omega\to \Rddsym$ satisfies $A(x)\in \partial_MH(x,D^2v(x))$ for a.e.\ $x\in\Omega $, then $A\in L^\infty(\Omega;\Rddsym)$ and $|A(x)|\leq L_H$ for a.e.\ $x\in \Omega$ by the uniform boundedness of the subdifferential sets $\partial_M H(x,M)$ for all $(x,M)\in \overline{\Omega}\times \Rddsym$ as shown above.
This immediately implies~\eqref{eq:DH_uniform_boundedness}.
It also shows that it is enough to show the existence of a measurable function $A$ satisfying  $A(x)\in \partial_MH(x,D^2v(x))$ for almost every $x\in\Omega $ in order to conclude that $\mathcal{D}H[v]$ is nonempty.
By~\cite[Theorem~10]{SmearsSuli2014}, there exists a Lebesgue measurable $\alpha_*\colon \Omega\to\mathcal{A}$ such that $\alpha_*(x)\in \Lambda(x,D^2 v(x))$ for a.e.\ $x\in \Omega$, where the set-valued map $\Lambda$ was defined in~\eqref{eq:Lambda_def} above. 
Then, by Proposition~\ref{prop:subdifferential_characterization}, we see that $-a(x,\alpha_*(x))\in \partial_M H(x,D^2v(x))$ for a.e.\ $x\in \Omega$, which shows that $\mathcal{D}H[v]$ is nonempty.
Finally, the convexity of $\mathcal{D}H[v]$ follows immediately from the convexity of the subdifferential sets.
\end{proof}

The following result shows that the graph of the set-valued map $\mathcal{D}H$ is sequentially closed in a suitable sense.

\begin{lemma}\label{lem:closure_measurable_selections}
Let $\{v_j\}_{j\in\N}$ be a sequence in $V$ and let $A_j\in \mathcal{D}H[v_j]$ for each $j\in \N$.
If $v_j\to v$ in $V$ and if $A_j\overset{*}{\rightharpoonup} A$ in $L^\infty(\Omega;\Rddsym)$ as $j\to \infty$ for some $v\in V$ and $A\in L^\infty(\Omega;\Rddsym)$, then $A\in \mathcal{D}H[v]$.
\end{lemma}
\begin{proof}
The proof is a straightforward extension of \cite[Lemma~4.4]{OsborneSmears2024i}.
\end{proof}
Observe that Lemma~\ref{lem:closure_measurable_selections} together with the boundedness property~\eqref{eq:DH_uniform_boundedness} imply that $\mathcal{D}H[v]$ is weak-$*$ sequentially compact in $L^\infty(\Omega;\Rddsym)$ for any $v\in V$.

\begin{remark}\label{rem:no_strong_continuity}
We emphasize the fact that, beyond Lemma~\ref{lem:closure_measurable_selections}, the set-valued map $\mathcal{D}H$ does not generally have good continuity properties, especially with regards to convergence in norm.
In general, it is possible to have sequences of functions $\{v_j\}_{j\in \N}$, with corresponding selections $A_j\in \mathcal{D}H[v_j]$ for all $j\in \N$, such that $\{A_j\}_{j\in\N}$ has no Cauchy subsequence in any Lebesgue space $L^p$ for any $p\geq 1$.
To give a simple example, consider $\Omega=(0,1)$ in one space dimension, and consider the Hamiltonian $H(x,M)=\max\{-M,-2M\}$, where we note that the Hessian variable $M\in \R$ for problems in one space dimension. 
Then, take the sequence defined by $v_j(x)= 2^{-3j}\sin(2^j \pi x)$ for $x\in(0,1)$ and $j\in\N$, and note that $v_j\tends 0$ in $H^2(\Omega)$ as $j\tends \infty$.
Then, for each $j\in \N$, the selections $A_j\in \mathcal{D}H[v_j]$ are uniquely determined up to sets of measure zero by $A_j(x)=-2$ if $\sin(2^j\pi x)\geq 0$ and $A_j(x)=-1$ otherwise. 
Thus the function~$A_j$ alternates between the values $-2$ and $-1$ over $2^j$ subintervals of $(0,1)$ of equal length, and since these subintervals are dyadically nested with respect to $j$, it is clear that the set $ \{x\in \Omega : |A_j(x)-A_{j^\prime}(x)|=1\}$ has measure $1/2$ for any $j$, $j^\prime \in \N$ such that $j\neq j^\prime$. This implies that $\norm{A_j-A_{j^\prime}}_{L^1(\Omega)}\geq 1/2$ whenever $j\neq j^\prime$, and thus also, by H\"older's inequality and the finiteness of the measure of~$\Omega$, that there is no subsequence of $\{A_j\}_{j\in\N}$ that is Cauchy in the $L^p$ norm for any $p\geq 1$.
\end{remark}

\subsection{Formulations of the problem}
\label{subsection:problem_formulations}
We now introduce the two problem formulations of the MFG system.
The first formulation, which we name the PDI problem, is the precise problem formulation for the PDI system~\eqref{eqn:mfg_pdi_system} discussed above. 
This formulation is the natural extension from~\cite{OsborneSmears2024i} to the fully nonlinear setting.
In the second formulation that we consider, the KFP equation is replaced by a variational inequality~(VI). 
We therefore refer to it as the VI problem.
As explained above, the motivation for considering the VI problem is flexibility with respect to passage to the limits, as will become clearer in the analysis below.
The theorem on equivalence of the two formulations is given in the next section.

\paragraph{Formulation of the PDI problem.}
First, we consider the MFG PDI problem: find $(u,m)\in V\times L^2(\Omega)$ such that there exists an $-\overline{a}\in \mathcal{D}H[u]$ for which
\begin{subequations}\label{eq:MFG_PDI}
\begin{alignat}{2}
H[u]&= F[m] &\quad& \text{in } \Omega, \label{eq:HJB_strong_form}\\
\left( m , -\overline{a}:D^2 v\right)_\Omega &= \pair{G}{v}_{V^*\times V} &\quad& \forall v\in V.
\label{eq:KFP_inclusion}
\end{alignat}
\end{subequations}
We shall refer to such a pair $(u,m) \in V \times L^2(\Omega)$ as a solution of the PDI problem, or, more informally, a~\emph{PDI solution}.
The HJB equation in~\eqref{eq:HJB_strong_form} is understood in the sense of strong solutions, i.e.\ $u\in V$ solves $H[u](x)=F[m](x)$ a.e.\ in $\Omega$, where we recall that $H[u](x)=H(x,D^2u(x))$ for $x\in\Omega$.
The KFP equation~\eqref{eq:KFP_inclusion}, which is given in very weak form, is obtained from~\eqref{eqn:mfg_pdi_2} after multiplication by a test function $v\in V$ and integrating-by-parts twice the second-order derivative terms, and using the homogeneous Dirichlet boundary conditions for both the test function $v$ and for $m$ to cancel the boundary terms.

\paragraph{Formulation of the VI problem.}\label{section: notions of solution}
Second, we define the VI problem: find $(u,m)\in V\times L^2_+(\Omega)$ such that
\begin{subequations}\label{eq:MFG_VI}
\begin{alignat}{2}
H[u]&= F[m] &\quad& \text{in } \Omega, \label{eq:HJB_strong_form_2}
\\ \pair{G}{v-u}_{V^*\times V} &\leq \left(m, H[v]-H[u]\right)_\Omega &\quad& \forall v \in V. \label{eq:KFP_VI} 
\end{alignat}
\end{subequations}
We say that a pair $(u,m)\in V\times L^2_+(\Omega)$ that satisfies~\eqref{eq:MFG_VI} is a solution of the VI problem, or, more informally, a \emph{VI solution}.
As before, the HJB equation $H[u]=F[m]$ in $\Omega$ is understood in the sense of strong solutions.
It is straightforward to derive the VI problem~\eqref{eq:MFG_VI} above from the PDI problem~\eqref{eq:MFG_PDI} when assuming the existence of a PDI solution $(u,m)\in V\times L^2_+(\Omega)$; the interested reader may skip ahead to Lemma~\ref{lem:PDI_solution_are_VI_solution} below.
The VI in~\eqref{eq:KFP_VI} can also be stated in abstract form as requiring that $G\in V^*$ must be contained in the subdifferential, evaluated at $u$, of the $m$-dependent convex functional $V\ni v\mapsto (m,H[v])_\Omega\in \R$.

We emphasize that the definition of a VI solution above includes the requirement that $m\in L^2_+(\Omega)$, i.e.\ $m\in L^2(\Omega)$ must also be nonnegative a.e.\ in $\Omega$.
This condition is natural since the density~$m$ should be nonnegative to be physically meaningful, and since the nonnegativity of~$m$ is used in the derivation of the VI problem~\eqref{eq:MFG_VI} from the PDI problem~\eqref{eq:MFG_PDI} that will be shown in Lemma~\ref{lem:PDI_solution_are_VI_solution} below.
Note that in the analysis below, the requirement that the player density~$m$ be nonnegative is compatible with the hypothesis that the source term~$G$ is nonnegative, as a result of the sign in front of the diffusion coefficient that appears in~\eqref{eqn:H_def} where the Hamiltonian was defined.
In contrast to the VI problem, the nonnegativity of the player density~$m$ is not a requirement of the definition of a solution of the PDI problem~\eqref{eq:MFG_PDI}.
It is instead usually deduced as a consequence of a comparison principle for KFP equations, under suitable assumptions on the problem data, e.g.\ if $G$ is assumed to be nonnegative, see Corollary~\ref{cor:kfp_comparison} below.

\begin{remark}[Relation to other variational inequalities]\label{rem:other_vi}
Variational inequalities for MFG systems have been analysed previously by Ferreira and Gomes in~\cite{FerreiraGomes2018}, and also with \"{U}\c{c}er in~\cite{FerreiraGomesUcer2025}, through the lens of monotone operator theory. They have obtained existence results for broad classes of MFG systems that have some global monotonicity structure.
Although their approach provides much insight and inspiration, for the avoidance of any confusion, we wish to clarify that the VI problem~\eqref{eq:MFG_VI} above, along with our approach to the analysis, is substantially different.
First, the operator defining the variational inequality in~\cite[Eq.~(1.7)]{FerreiraGomes2018} requires significantly more differentiability of the Hamiltonian, in particular that $H$ must be differentiable in $M$ with $\frac{\partial H}{\partial M}$ additionally at least $C^2$ regular with respect to the Hessian variable~\cite[Section~7]{FerreiraGomes2018}.
By contrast, the VI problem~\eqref{eq:MFG_VI} is still meaningful for nondifferentiable Hamiltonians.
Second, the HJB equation in~\eqref{eq:MFG_VI} is understood in the sense of strong solutions, rather than in the form of a variational inequality as is the case in~\cite[Eq.~(1.7)]{FerreiraGomes2018}.
Finally, the analysis of existence of solutions that follows below does not require monotonicity of the coupling operator $F$.
On the other hand, we only focus here on the fully nonlinear second-order setting, whereas the references above also consider first-order problems.
\end{remark}

\subsection{Cordes coefficients}

The analysis of existence and uniqueness of solutions of the PDI and VI problems stated above will focus on the class of problems with Cordes coefficients, as we now detail.
Recall that $\Rddsymp$ denotes the set of symmetric positive semidefinite matrices in $\R^{d\times d}$.
For positive constants $\overline{\nu} \geq \underline{\nu}>0$, and $\varepsilon\in (0,1]$, we define~$\Cclass$ to be the set of all matrices $\overline{a}\in \Rddsymp$ that satisfy the uniform ellipticity condition
\begin{equation}\label{eq:uniform_ellipticity}
\begin{aligned}
\underline{\nu} |\xi|^2 \leq \xi^T\, \overline{a} \,\xi  \leq \overline{\nu} |\xi|^2 &&& \forall \xi \in \R^d,
\end{aligned}
\end{equation}
as well as the Cordes condition
\begin{equation}\label{eq:Cordes_condition}
\frac{\abs{\overline{a}}}{\Trace \overline{a}} \leq \frac{1}{\sqrt{d-1+\varepsilon}},
\end{equation}
where it is recalled that~$\abs{\,\cdot\,}$ denotes the Frobenius norm of matrices in $\Rdd$.
Note that $\Cclass$ is a closed convex subset of~$\Rddsymp$. 
Indeed, it is clear that~$\Cclass$ is closed.
To check that~$\Cclass$ is convex, note that $\lambda a_1+ (1-\lambda) a_2\in \Rddsymp$ also satisfies the uniform ellipticity condition~\eqref{eq:uniform_ellipticity} for any $a_1$, $a_2 \in \Cclass$, and any $\lambda\in [0,1]$.  
Furthermore, the triangle inequality implies that
\begin{equation}\label{eq:Cordes_1}
\abs{ \lambda a_1 + (1-\lambda) a_2}\leq \lambda \abs{a_1} + (1-\lambda)\abs{a_2} \leq  
\frac{\Trace \left(\lambda a_1 + (1-\lambda) a_2\right)}{\sqrt{d-1+\varepsilon}}.
\end{equation}
The uniform ellipticity condition~\eqref{eq:uniform_ellipticity} implies that $\Trace \left(\lambda a_1 + (1-\lambda) a_2\right)>0$, and thus we conclude from~\eqref{eq:Cordes_1} that $\lambda a_1 + (1-\lambda) a_2$ also satisfies~\eqref{eq:Cordes_condition}, and thus also belongs to $\Cclass$. 
This shows that $\Cclass$ is convex.

The principal assumption that we will use in the analysis of existence and uniqueness of solutions of the MFG PDI and VI problems is that the problem has Cordes coefficients. More precisely, we make the following assumption.
\begin{assumption}\label{assumption: Cordes}
There exist constants $\overline{\nu}\geq\underline{\nu}>0$, and $\varepsilon \in (0,1]$ such that $a(x,\alpha)\in \Cclass$ for all $x\in\overline{\Omega}$ and all $\alpha\in\mathcal{A}$.
\end{assumption}

\begin{remark}
For problems in one or two space dimensions, i.e.\ $d \leq 2$, the uniform ellipticity condition~\eqref{eq:uniform_ellipticity} implies the Cordes condition~\eqref{eq:Cordes_condition} for some $\varepsilon=\varepsilon(\underline{\nu},\overline{\nu})\in (0,1]$. Indeed, the case $d=1$ is obvious, and the sufficiency of uniform ellipticity for the case $d=2$ is shown in~\cite[Example 2]{SmearsSuli2014}.
The condition~\eqref{eq:Cordes_condition} only becomes more restrictive when $d\geq 3$.
Geometrically, the Cordes condition can be seen as defining a particular cone in the space~$\Rddsym$ of symmetric matrices with vertex at the origin and principal axis passing through the identity matrix.
\end{remark}

\begin{remark}
Following on from the previous remark, notice that uniformly elliptic Hamiltonians that depend only on the Laplacian $\Delta u$, rather than the full Hessian $D^2 u$, are special cases of the class of Hamiltonians considered here.
\end{remark}

\begin{remark}\label{rem:Cordes_condition_original}
In~\cite{Cordes1956}, Cordes formulated his condition in terms of the eigenvalues $\{\lambda_i\}_{i=1}^d$ of the diffusion coefficient of the elliptic operator. 
In particular, after adapting slightly the notation, the condition in \cite[Definition~1]{Cordes1956} is written as $(d-1)\sum_{i<k}(\lambda_i-\lambda_k)^2\leq (1-\widetilde{\epsilon})\left(\sum_{i} \lambda_i\right)^2$ for some positive $\widetilde{\epsilon}$, where $\sum_{i<k}$ denotes a double sum over all indices $(i,k)\in\{1,\ldots,d\}^2$ with $i<k$. For $\overline{a}\in\Rddsymp$ with eigenvalues $\{\lambda_i\}_{i=1}^d$, we have the elementary identities $\Trace \overline{a} = \sum_{i} \lambda_i$ and $\sum_{i<k}(\lambda_i-\lambda_k)^2=d \abs{\overline{a}}^2-\left(\Trace \overline{a}\right)^2$, thus the condition there is equivalent to $\frac{\abs{\overline{a}}^2}{(\Trace\overline{a})^2}\leq \frac{d-\widetilde{\epsilon}}{d(d-1)}$, and since $d-\widetilde{\epsilon}<d$, we see that the above condition is equivalent to~\eqref{eq:Cordes_condition} for some~$\varepsilon$ depending on~$\widetilde{\epsilon}$.
\end{remark}

An important, although elementary, consequence of the convexity of the set $\Cclass$ is that the uniform ellipticity and the Cordes condition are preserved upon taking measurable selections from the subdifferential sets of the Hamiltonian.

\begin{lemma}\label{lem:Cordes_subdifferential}
Suppose that Assumption~\ref{assumption: Cordes} holds.
Then, for any $v\in V$ and any $-\overline{a}\in\mathcal{D}H[v]$, we have $\overline{a}(x)\in \Cclass$ for a.e.\ $x\in\Omega$.
\end{lemma}
\begin{proof}
Recall that $\Cclass$ is a closed convex set in $\Rddsym$, and recall also Proposition~\ref{prop:subdifferential_characterization}, in particular~\eqref{eq:subdifferential_characterization}, which shows that $\partial_MH(x,M)$ is the convex hull of the set $\{-a(x,\alpha) : \alpha\in \Lambda(x,M)\}$.
Recall also that the convex hull of a set is the smallest convex set that contains the given set.
Therefore, by~\eqref{eq:subdifferential_characterization} we deduce that $\partial_M H(x,M)\subset -\Cclass$ for any $(x,M)\in \overline{\Omega}\times \Rddsym$.
In other words, for any $(x,M)\in\overline{\Omega}\times\Rddsym$ and any $\overline{a}\in\Rddsym$ such that $-\overline{a}\in\partial_M H(x,M)$, we have  $\overline{a}\in\Cclass$.
This completes the proof, upon recalling that, by definition, any $-\overline{a}\in\mathcal{D}H[v]$ must satisfy $-\overline{a}(x)\in \partial_M H(x,D^2v(x))$ for almost every $x\in \Omega$.
\end{proof}

\section{Main results}\label{sec:main_results}

\subsection{Equivalence between the PDI and VI problems}

The first main result of this work is that the concepts of PDI solution and VI solution coincide for solution pairs where the player density is nonnegative.

\begin{theorem}[Equivalence of PDI and VI solutions]\label{thm:PDI_VI_equivalence}
Let $\Omega$ be a nonempty bounded open subset of $\R^d$, $d\geq 1$.
Let $\mathcal{A}$ be a compact metric space, let $a\in C(\overline{\Omega}\times \mathcal{A};\Rddsymp)$, and let $f\in C(\overline{\Omega}\times \mathcal{A})$.
If $(u,m)\in V\times L^2(\Omega) $ is a solution of the PDI problem~\eqref{eq:MFG_PDI} and if $m\geq 0$ a.e.\ in $\Omega$, then $(u,m)$ is a solution of the VI problem~\eqref{eq:MFG_VI}.
Conversely, if $(u,m)\in V\times L^2_+(\Omega)$ is a solution of the VI problem~\eqref{eq:MFG_VI}, then $(u,m)$ is a solution of the PDI problem~\eqref{eq:MFG_PDI}.
\end{theorem}

The proof of Theorem~\ref{thm:PDI_VI_equivalence} is given in Section~\ref{section: equivalence proof} below.
Observe that, in Theorem~\ref{thm:PDI_VI_equivalence}, a PDI solution $(u,m)$ is shown to be a VI solution if $m$ is additionally assumed to be nonnegative in~$\Omega$.
In practice, the nonnegativity of $m$ will be obtained by means of a comparison principle, see Section~\ref{sec:comparison} below.
In particular, under the further conditions that $\Omega$ is convex, Assumption~\ref{assumption: Cordes} holds, and that $G\in V^*$ is nonnegative, the equivalence of the notions of PDI and VI solutions is complete.

\begin{corollary}\label{cor:PDI_VI_equivalence_complete}
Let~$\Omega$ be a nonempty bounded open convex subset of~$\R^d$, $d\geq 1$. 
Let~$\mathcal{A}$ be a compact metric space, let $a\in C( \overline{\Omega}\times \mathcal{A};\Rddsymp)$ satisfy Assumption~\ref{assumption: Cordes}, and let $f\in C(\overline{\Omega}\times \mathcal{A})$.
Let $G\in V^*$ be nonnegative.
Then a pair $(u,m)$ is a solution of the PDI problem~\eqref{eq:MFG_PDI} if and only if $(u,m)$ is a solution of the VI problem~\eqref{eq:MFG_VI}.
\end{corollary}

The proof of Corollary~\ref{cor:PDI_VI_equivalence_complete} is given in Section~\ref{sec:comparison} after showing a suitable comparison principle for solutions of the KFP equation.

\begin{remark}
We emphasize that Theorem~\ref{thm:PDI_VI_equivalence} does not require many of the other assumptions that will be used later in this work.
In particular, Theorem~\ref{thm:PDI_VI_equivalence} does not require the assumptions of uniform ellipticity or the Cordes condition in Assumption~\ref{assumption: Cordes}, which only play a role in the comparison principle for the KFP equation in the proof of Corollary~\ref{cor:PDI_VI_equivalence_complete}, and the later existence and uniqueness results.
The proof of Theorem~\ref{thm:PDI_VI_equivalence} also does not require any properties of the coupling term $F$, nor does it require convexity of the domain~$\Omega$.
As will soon become apparent, the main ingredients in the proof of Theorem~\ref{thm:PDI_VI_equivalence} are merely the nonemptiness, convexity, and uniform boundedness of the sets $\mathcal{D}H[v]$ for $v\in V$, along with the closed graph property expressed in Lemma~\ref{lem:closure_measurable_selections}.
We therefore expect that the VI formulation of MFG systems may be a useful tool in many other contexts and other functional settings.
\end{remark}

\subsection{Existence of a solution}

The second main result of this work is the existence of a PDI solution and, equivalently, a VI solution for the problem, still under the hypotheses that $\Omega$ is convex, that $G$ is nonnegative, and that Assumption~\ref{assumption: Cordes} holds.
Since Corollary~\ref{cor:PDI_VI_equivalence_complete} holds under these hypotheses, the notions of PDI and VI solution coincide.
The principal additional assumption to those above for existence of a solution is that $F\colon L^2(\Omega)\to L^2(\Omega)$ is completely continuous.

\begin{theorem}[Existence of a solution]
\label{thm: MFG existence}
Let~$\Omega$ be a nonempty bounded open convex subset of~$\R^d$, $d\geq 1$. 
Let~$\mathcal{A}$ be a compact metric space, let $a\in C( \overline{\Omega}\times \mathcal{A};\Rddsymp)$ satisfy Assumption~\ref{assumption: Cordes}, and let $f\in C(\overline{\Omega}\times \mathcal{A})$.
Let $F\colon L^2(\Omega)\to L^2(\Omega)$ be completely continuous.
Let $G\in V^*$ be nonnegative.
Then there exists a solution $(u,m)\in V\times L^2_+(\Omega)$ of the PDI problem~\eqref{eq:MFG_PDI} and, equivalently, of the VI problem~\eqref{eq:MFG_VI}.
Furthermore, any solution $(u,m)$ satisfies the \emph{a priori} bounds
\begin{subequations}\label{eq:solution_boundedness}
    \begin{align}
    \norm{m}_{L^2(\Omega)}& \leq C \norm{G}_{V^*}, \label{eq:solution_boundedness_m}\\
    \norm{u}_{H^2(\Omega)} &\leq C, \label{eq:solution_boundedness_u}
    \end{align}
\end{subequations}
where the above constants depend only on $d$, $\diam(\Omega)$, $\varepsilon$, $\underline{\nu}$, $\overline{\nu}$, and, in the case of~\eqref{eq:solution_boundedness_u}, additionally on $F$, $\|f\|_{C(\overline{\Omega}\times \mathcal{A})}$, and on $\norm{G}_{V^*}$.
\end{theorem}

The proof of Theorem~\ref{thm: MFG existence} is given in Section~\ref{section: well-posed}.

\begin{remark}
We leave the form of the dependence on $F$ and $\norm{G}_{V^*}$ of the constant of the \emph{a priori} bound~\eqref{eq:solution_boundedness_u} implicit since we have not specified a growth condition on $F$.
However, the bound can be made explicit if $F$ obeys some specific growth condition, since one can then simply use the bound~\eqref{eq:HJB_boundedness} from Lemma~\ref{lemma: HJB existence} below.
For example, if one assumes a growth condition of the form $\norm{F[m]}_{L^2(\Omega)}\leq C\left( 1 + \norm{m}_{L^2(\Omega)}^\kappa\right)$ for some $\kappa \geq 0$, then one can show that $\norm{u}_{H^2(\Omega)}\leq C\left(1+\norm{G}_{V^*}^\kappa\right)$, for a constant that depends only on $d$, $\diam(\Omega)$, $\varepsilon$, $\underline{\nu}$, $\overline{\nu}$, $F$ and~$\|f\|_{C(\overline{\Omega}\times \mathcal{A})}$. 
See Example~\ref{example:strictly_monotone_operator} above for some examples of operators satisfying such a growth condition.
\end{remark}

\subsection{Uniqueness of the solution}\label{sec:uniqueness_theorem}

Similar to the results on existence of solutions, we place ourselves again in the setting of Corollary~\ref{cor:PDI_VI_equivalence_complete}, whereby the notions of PDI and VI solution coincide.
If the operator~$F$ additionally satisfies the strict monotonicity assumption~\ref{assumption: strict monotonicity}, then we are able to show the uniqueness of a solution for the PDI problem~\eqref{eq:MFG_PDI}, and equivalently of the VI problem~\eqref{eq:MFG_VI}, even though~$H$ is not required to be differentiable.

\begin{theorem}[Uniqueness of the solution]\label{thm:MFG_uniqueness}
Let $\Omega$ be a nonempty bounded open convex subset of $\R^d$, $d\geq 1$.
Let $\mathcal{A}$ be a compact metric space, let $a\in C(\overline{\Omega}\times \mathcal{A};\Rddsymp)$ satisfy Assumption~\ref{assumption: Cordes}, and let $f\in C(\overline{\Omega}\times \mathcal{A})$.
Let $G\in V^*$ be nonnegative.
If the operator $F\colon L^2(\Omega)\to L^2(\Omega)$ satisfies Assumption~\ref{assumption: strict monotonicity}, then any solution $(u,m) \in V \times L^2(\Omega)$ of the PDI problem~\eqref{eq:MFG_PDI} and, equivalently, of the VI problem~\eqref{eq:MFG_VI}, is unique.
\end{theorem}

The proof of Theorem~\ref{thm:MFG_uniqueness} is given in Section~\ref{sec:uniqueness} below.
The first step towards its proof is to show that any two pairs of solutions necessarily have the same player density. 
The argument is essentially the well-known one of Lasry and Lions~\cite{LasryLions2006i,LasryLions2006ii,LasryLions2007}, which takes an especially simple and elegant form when working with VI solutions, as we immediately demonstrate below.

\begin{lemma}\label{lem:VI_density_uniqueness}
Suppose that the operator $F\colon L^2(\Omega)\to L^2(\Omega)$ satisfies Assumption~\ref{assumption: strict monotonicity}. 
If $(u_i, m_i) \in V \times L^2_+(\Omega)$, $i\in\{1,2\}$, are both solutions of the VI problem~\eqref{eq:MFG_VI}, then $m_1=m_2$.
\end{lemma}
\begin{proof}
Under the hypotheses above, we have
\begin{equation}
\begin{split}
0 & \hspace{8pt}=\hspace{8pt}\angled{G, u_2 - u_1}_{V^*\times V} + \angled{G, u_1 - u_2}_{V^*\times V} 
\\ &\underset{\eqref{eq:KFP_VI}}{\leq} (m_1,H[u_2]-H[u_1])_\Omega+(m_2,H[u_1]-H[u_2])_\Omega
\\ & \underset{\eqref{eq:HJB_strong_form_2}}{=} - (F[m_1]-F[m_2],m_1-m_2)_\Omega.
\end{split}
\end{equation}
Thus $(F[m_1]-F[m_2],m_1-m_2)_\Omega\leq 0$ which implies that $m_1=m_2$ by Assumption~\ref{assumption: strict monotonicity}.
\end{proof}

The remainder of the proof of Theorem~\ref{thm:MFG_uniqueness}, which is detailed in Section~\ref{sec:uniqueness} below, then consists of showing that any pair of VI solutions must also have the same value function, which will be obtained from the uniqueness of strong solutions in $V$ of HJB equations with Cordes coefficients.

\begin{remark}
The uniqueness result of Theorem~\ref{thm:MFG_uniqueness} above does not suppose, or imply, the uniqueness of elements from $\mathcal{D}H[u]$ that appear in the KFP inclusion.
Indeed, it is known from the analysis of MFG PDI systems in the quasilinear setting that the uniqueness of these measurable selections is not generally guaranteed. 
See~\cite{OsborneSmears2024i} for further discussion and examples.
As a further point, it is also known from examples in the quasilinear setting~\cite{OsborneSmears2025iii} that uniqueness of the solution is not expected if the coupling $F$ is merely monotone, rather than strictly monotone.
A brief explanation is that the case $F\equiv 0$, which is monotone but not strictly monotone, decouples the problem, resulting in a game of uncoupled players with possibly nonunique optimal controls, for which we cannot generally expect uniqueness of the player density.
\end{remark}

Thus, in light of the results above, we deduce the existence and uniqueness of a solution of the PDI problem~\eqref{eq:MFG_PDI} and, equivalently, of the VI problem~\eqref{eq:MFG_VI}, whenever the hypotheses of Theorems~\ref{thm: MFG existence} and~\ref{thm:MFG_uniqueness} are satisfied.

\begin{remark}
It is possible to drop the hypothesis that $G$ should be nonnegative in Theorem~\ref{thm:MFG_uniqueness} if one is only interested in the uniqueness of solutions of the VI problem~\eqref{eq:MFG_VI}.
This is because the nonnegativity of the player density is required as part of the definition of a VI solution.
\end{remark}

\section{Equivalence of the PDI and VI problems}\label{section: equivalence proof}

In this section, we prove the equivalence of the PDI and VI problems, as stated in Theorem~\ref{thm:PDI_VI_equivalence}. 
Throughout this section, the existence of a solution is assumed, given that the existence of solutions is the subject of Theorem~\ref{thm: MFG existence}.
The proof that PDI solutions with nonnegative densities are also VI solutions is quite straightforward, as we now show.

\begin{lemma}\label{lem:PDI_solution_are_VI_solution}
Suppose that $(u,m)\in V\times L^2(\Omega)$ is a solution of the PDI problem~\eqref{eq:MFG_PDI} and that $m\geq 0$ a.e.\ in $\Omega$.
Then $(u,m)$ is a solution of the VI problem~\eqref{eq:MFG_VI}.
\end{lemma}
\begin{proof}
By the hypothesis that $(u,m)$ solves~\eqref{eq:MFG_PDI}, there exists an $\overline{a}\in L^\infty(\Omega;\Rddsym)$ such that $-\overline{a}\in \mathcal{D}H[u]$ and $\pair{G}{v}_{V^*\times V} = \left(m, -\overline{a}:D^2 v\right)_\Omega$ for all $v\in V$. 
Let $w\in V$ be arbitrary and set $v=w-u$.
Since $-\overline{a} \in \mathcal{D}H[u]$, it follows from the definition of the subdifferential and from $-\overline{a}(x)\in \partial_M H(x,D^2 u(x))$ for a.e.\ $x\in\Omega$ that
\[
\begin{aligned}
-\overline{a}:D^2 v &= -\overline{a}:D^2(w-u) \leq H[w]-H[u] &&\text{a.e.\ in }\Omega.
\end{aligned}
\] 
The hypotheses that $m\in L^2(\Omega)$ and $m\geq 0$ a.e.\ in $\Omega$ then imply that
\begin{equation*}
\pair{G}{w-u}_{V^*\times V} = \left(m,-\overline{a}:D^2(w-u)\right)_\Omega \leq  \left( m , H[w]-H[u]\right)_\Omega.
\end{equation*} 
Since $w\in V$ is arbitrary, this shows that $(u,m)$ is a solution of the VI problem~\eqref{eq:MFG_VI}.
\end{proof}

Our next goal is to show that the converse statement is true, i.e.\ any VI solution is a PDI solution.
Thus, if we suppose that $(u,m)$ is a given VI solution, the task is to show the existence of an element $-\overline{a}\in \mathcal{D}H[u]$ such that~\eqref{eq:KFP_inclusion} holds.
The first step in this direction is the following lemma.
\begin{lemma}\label{lem:VI_directional_selection}
Suppose that $(u,m)\in V\times L^2_+(\Omega)$ is a solution of the VI problem~\eqref{eq:MFG_VI}.
Then, for each $v\in V$, there exists an $\overline{a}$ such that $-\overline{a}\in \mathcal{D}H[u]$ and
\begin{equation}\label{eq:VI_directional_selection_0}
    \left( m, -\overline{a}:D^2 v \right)_\Omega = \pair{G}{v}_{V^*\times V}.
\end{equation}
\end{lemma}
\begin{proof}
Let $v\in V$ be given. For each integer $j\in \N$, we define
\begin{equation}
w_{+j} \coloneqq u + \frac{1}{j} v, \quad w_{-j} \coloneqq u - \frac{1}{j}v.
\end{equation}
By hypothesis, $(u,m)$ satisfies~\eqref{eq:KFP_VI}, so $\pair{G}{w_{\pm j} - u}_{V^*\times V}\leq \left( m ,  H[w_{\pm j}]-H[u]\right)_\Omega$ for all $ j\in \N$.
Choose an arbitrary element $-\overline{a}_{\pm j}\in \mathcal{D}H[w_{\pm j} ] $ for each $j\in \N$. 
Then, we have $H[w_{\pm j}]-H[u]\leq -\overline{a}_{\pm j}:D^2(w_{\pm j}-u)$ a.e.\ in $\Omega$, for all $j\in \N$.
Since $m\geq 0$ a.e.\ in $\Omega$, it follows that
\begin{align*}
 \frac{1}{j}\pair{G}{v}_{V^*\times V} &\leq \frac{1}{j} \left( m, -\overline{a}_{+j}:D^2v \right)_\Omega,
&
\frac{1}{j}\pair{G}{-v}_{V^*\times V} &\leq \frac{1}{j}\left( m , -\overline{a}_{-j}:D^2(-v) \right)_\Omega ,
\end{align*}
for each $j\in \N$.
Thus, after simplification, we obtain
\begin{equation}\label{eq:VI_directional_selection_1}
\begin{aligned}
\left( m , -\overline{a}_{-j}:D^2 v \right)_\Omega  \leq  \pair{G}{v}_{V^*\times V} \leq \left( m, -\overline{a}_{+j}:D^2v \right)_\Omega &&& \forall j\in \N.
\end{aligned}
\end{equation}
Since $w_{\pm j} \to u$ in $V$ as $j\to \infty$, and since $-\overline{a}_{\pm j}\in \mathcal{D}H[w_{\pm j}]$ for each $j\in \N$, we deduce from Lemma~\ref{lem:closure_measurable_selections} and from the uniform boundedness of the sets $\mathcal{D}H[w_{\pm j}]$, see~\eqref{eq:DH_uniform_boundedness} above, that there exist elements $-\overline{a}_+\in \mathcal{D}H[u]$ and $-\overline{a}_-\in \mathcal{D}H[u]$, and a subsequence such that $\overline{a}_{+j_k}\overset{*}{\rightharpoonup} \overline{a}_+$ and $\overline{a}_{-j_k} \overset{*}{\rightharpoonup} \overline{a}_-$ in $L^\infty(\Omega;\Rddsym)$ as $k\to \infty$.
We can then pass to the limit in~\eqref{eq:VI_directional_selection_1} to find that
\begin{equation}\label{eq:VI_directional_selection_2}
\left(m, -\overline{a}_-:D^2 v \right)_\Omega \leq \pair{G}{v}_{V^*\times V} \leq  \left( m,  -\overline{a}_+:D^2 v\right)_\Omega.
\end{equation}
For each $\lambda\in[0,1]$, let $\overline{a}_{\lambda}\coloneqq \lambda \overline{a}_+ + (1-\lambda)\overline{a}_-$.
Convexity of $\mathcal{D}H[u]$ implies that $-\overline{a}_\lambda\in \mathcal{D}H[u]$ for all $\lambda\in[0,1]$.
Combining the inequalities in~\eqref{eq:VI_directional_selection_2} with the intermediate value theorem, we conclude that there exists a $\lambda \in [0,1]$ such that $\left(m, - \overline{a}_{\lambda}:D^2 v \right)_\Omega = \pair{G}{v}_{V^*\times V}$.
This completes the proof of the lemma.
\end{proof}

\begin{remark}
We emphasize that, in Lemma~\ref{lem:VI_directional_selection}, the element $\overline{a}$ appearing in~\eqref{eq:VI_directional_selection_0} is allowed to depend on the given test function $v$. 
Thus, Lemma~\ref{lem:VI_directional_selection} does not immediately verify~\eqref{eq:KFP_inclusion}.
However, if one is only interested in the case where $H$ is assumed to be differentiable, in which case $\mathcal{D}H[u]$ is a singleton set, then~\eqref{eq:KFP_inclusion} can be deduced immediately from Lemma~\ref{lem:VI_directional_selection}.
\end{remark}

In the proof of Theorem~\ref{thm:PDI_VI_equivalence} that is given shortly below, we will use a simple well-known property of Hilbert spaces.
Recall that the space $V$ is a Hilbert space with the $H^2$-inner-product, which we denote by~$(\cdot, \cdot)_V$, so $\norm{v}_{H^2(\Omega)}^2=(v,v)_V$ for all $v\in V$.
The Riesz representation theorem shows that, for each bounded linear functional $\ell\in V^*$, there exists a unique $v_\ell\in V$ such that $(v_\ell,v)_V=\ell(v)$ for all $v\in V$, and we have the isometry identity
\begin{equation}\label{eq:Riesz_isometry}
\norm{v_\ell}_{H^2(\Omega)}  = \norm{\ell}_{V^*} \coloneqq \max_{v\in B_V} \ell(v),
\end{equation}
where $B_V\coloneqq \{v\in V:\norm{v}_{H^2(\Omega)}\leq 1\}$ denotes the closed unit ball in $V$. 
Note that the maximum of~$\ell$ is achieved over $B_V$ since~$\ell \in V^*$ and $B_V$ is weakly sequentially compact as $V$ is reflexive.
To show the first equality in~\eqref{eq:Riesz_isometry}, note that the Cauchy--Schwarz inequality implies that $\ell(v)\leq \norm{v_\ell}_V\norm{v}_V$ for all $v\in V$, and thus $\norm{\ell}_{V^*}\leq \norm{v_\ell}_V$. 
Conversely, observe that $\norm{v_\ell}_{H^2(\Omega)}^2=(v_\ell,v_\ell)_V=\ell(v_\ell)\leq \norm{\ell}_{V^*}\norm{v_\ell}_V$ so $\norm{v_\ell}_V\leq \norm{\ell}_{V^*}$.
Together, these inequalities imply~\eqref{eq:Riesz_isometry}.
We now prove Theorem~\ref{thm:PDI_VI_equivalence} on the equivalence between the notions of PDI and VI solutions.

\begin{proof}[Proof of Theorem~\ref{thm:PDI_VI_equivalence}]
Suppose that $(u,m)\in V\times L^2_+(\Omega)$ is a given solution of the VI problem~\eqref{eq:MFG_VI}.
Let the functional $\mathcal{L}\colon \mathcal{D}H[u]\times V\to\R$ be defined by
\begin{equation}\label{eq:PDI_VI_equivalence_0}
\mathcal{L}(A,v)\coloneqq \left( m, A:D^2 v\right)_\Omega  -  \pair{G}{v}_{V^*\times V} \quad \forall (A,v)\in \mathcal{D}H[u]\times V.
\end{equation}
The functional $\mathcal{L}$ is affine in its first argument and linear in its second argument.
Next, we define the functional $\mathcal{J} \colon \mathcal{D}H[u] \to \R_{\geq 0}$ by 
\begin{equation}\label{eq:PDI_VI_equivalence_1}
\mathcal{J}(A) \coloneqq \max_{v\in B_V} \mathcal{L}(A,v) \quad \forall A \in \mathcal{D}H[u],
\end{equation}
where we recall that~$B_V$ denotes the closed unit ball in $V$.
In other words, $\mathcal{J}(A)$ is simply the $V^*$-norm of the bounded linear functional $\mathcal{L}(A,\cdot)\in V^*$, and we recall that the maximum is achieved in~\eqref{eq:PDI_VI_equivalence_1} since $B_V$ is weakly sequentially compact as $V$ is reflexive.
Furthermore, it is also clear that $\mathcal{J}$ is nonnegative since $B_V$ contains the origin and $\mathcal{L}(A,0)=0$ for any $A\in \mathcal{D}H[u]$.
Note also that the boundedness of $\mathcal{D}H[u]$ in $L^\infty(\Omega;\Rddsym)$, see~\eqref{eq:DH_uniform_boundedness}, implies that $\mathcal{J}$ is also uniformly bounded on $\mathcal{D}H[u]$.
Recall that $\mathcal{D}H[u]$ is convex and weak-$*$ sequentially compact as a subset of $L^\infty(\Omega;\Rddsym)$ by Lemma~\ref{lem:closure_measurable_selections}.
We now show that $\mathcal{J}$ is weak-$*$ sequentially lower-semicontinuous on $\mathcal{D}H[u]$, i.e.\ if a sequence $\{A_j\}_{j\in\N}\subset \mathcal{D}H[u]$ weak-$*$ converges to some $A\in\mathcal{D}H[u]$, then 
\begin{equation}\label{eq:J_weakstarlsc}
\mathcal{J}(A)\leq \liminf_{j\tends \infty} \mathcal{J}(A_j).
\end{equation}
To show~\eqref{eq:J_weakstarlsc}, note that the weak-$*$ convergence $A_j\overset{*}{\rightharpoonup} A$ as $j\tends \infty$ implies that
\begin{equation}\label{eq:J_weakstarlsc_1}
\begin{aligned}
\mathcal{L}(A,v)=\lim_{j\tends \infty}\mathcal{L}(A_j,v)\leq \liminf_{j\tends \infty} \mathcal{J}(A_j) &&& \forall v\in B_V,
\end{aligned}
\end{equation}
where we have used the fact that $\mathcal{L}(A_j,v)\leq \mathcal{J}(A_j)$ for all $v\in B_V$ and $j\in \N$, by definition of the functional~$\mathcal{J}$ in~\eqref{eq:PDI_VI_equivalence_1}.
We then obtain~\eqref{eq:J_weakstarlsc} from~\eqref{eq:J_weakstarlsc_1} by taking the supremum of $\mathcal{L}(A,v)$ over all $v\in B_V$.
The weak-$*$ sequential compactness of $\mathcal{D}H[u]$ and weak-$*$ sequential lower semicontinuity of~$\mathcal{J}$ shown above imply that there exists an $A_0\in \mathcal{D}H[u]$ that minimizes $\mathcal{J}$ over $\mathcal{D}H[u]$, i.e.\
\begin{equation}\label{eq:PDI_VI_equivalence_2}
\mathcal{J}(A_0) = \min_{A\in \mathcal{D}H[u]} \mathcal{J}(A).
\end{equation}
We will show shortly below that $\mathcal{J}(A_0)=0$, which will be sufficient to conclude that $(u,m)$ is a solution of the PDI problem~\eqref{eq:MFG_PDI}.
Indeed, once we have shown that $\mathcal{J}(A_0)=0$, we deduce that $\mathcal{L}(A_0,v)\leq 0$ for all $v\in B_V$ from the definition of~$\mathcal{J}$.
By linearity of $v\mapsto \mathcal{L}(A_0,v)$, this is equivalent to $\mathcal{L}(A_0,v)=0$ for all $v\in V$.
After setting $\overline{a} \coloneqq -A_0$, we have $(m , -\overline{a}:D^2 v)_{\Omega} = \pair{G}{v}_{V^*\times V}$ for all $v\in V$, which shows that $m$ solves~\eqref{eq:KFP_inclusion} for some $-\overline{a}\in \mathcal{D}H[u]$.
Thus it is enough to show that $\mathcal{J}(A_0)=0$ to deduce that the pair $(u,m)$ solves the PDI problem~\eqref{eq:MFG_PDI}.

We now show that $\mathcal{J}(A_0)=0$.
Since $V$ is a Hilbert space and since the map $v\mapsto \mathcal{L}(A_0,v)$ is a bounded linear functional on $V$, the Riesz representation theorem shows that there exists a unique $v_0\in V$ such that $(v_0,v)_V=\mathcal{L}(A_0,v)$ for all $v\in V$, where $(\cdot, \cdot)_V$ denotes the $H^2$-inner product on $V$.
We now apply Lemma~\ref{lem:VI_directional_selection} to find an $A_1\in \mathcal{D}H[u]$ such that
\begin{equation}\label{eq:PDI_VI_equivalence_5}
   \mathcal{L}(A_1,v_0)=\left( m, A_1 : D^2 v_0\right)_\Omega   - \pair{G}{v_0}_{V^*\times V} =0.
\end{equation}
Let $A_\lambda \coloneqq (1-\lambda)A_0+\lambda A_1$ for all $\lambda\in[0,1]$, and note that $A_\lambda\in \mathcal{D}H[u]$ since $\mathcal{D}H[u]$ is convex.
For each $\lambda\in [0,1]$, let $v_\lambda$ denote the unique element in $V$ that satisfies 
\begin{equation}\label{eq:PDI_VI_equivalence_6}
(v_\lambda ,v)_V = \mathcal{L}(A_\lambda,v)\quad \forall v\in V.
\end{equation}
By choosing $\lambda=1$ and the test function $v=v_0$ in~\eqref{eq:PDI_VI_equivalence_6} above, we deduce from~\eqref{eq:PDI_VI_equivalence_5} that $(v_1,v_0)_V=0$, i.e.\ $v_0$ and $v_1$ are orthogonal in $V$.
Note also that~\eqref{eq:PDI_VI_equivalence_6} and the Riesz isometry identity~\eqref{eq:Riesz_isometry} show that 
\begin{equation}\label{eq:PDI_VI_equivalence_10}
  \norm{v_\lambda}_{H^2(\Omega)}= \max_{v\in B_V}\mathcal{L}(A_\lambda,v) =  \mathcal{J}(A_\lambda) \quad \forall \lambda \in[0,1] .
\end{equation}
It is furthermore clear that $v_\lambda=(1-\lambda)v_0+\lambda v_1$ for all  $\lambda\in[0,1]$  since  $\mathcal{L}(A_\lambda,v)=(1-\lambda)\mathcal{L}(A_0,v)+\lambda\mathcal{L}(A_1,v)$ for all $v\in V$.
Therefore, by expanding the square, we find that
\begin{equation}\label{eq:PDI_VI_equivalence_8}
    \begin{split}
\mathcal{J}(A_\lambda)^2 & = \norm{v_{\lambda}}^2_{H^2(\Omega)} = \norm{(1-\lambda)v_0+\lambda v_1}_{H^2(\Omega)}^2
\\ & = (1-\lambda)^2 \norm{v_0}_{H^2(\Omega)}^2 + 2(1-\lambda)\lambda (v_1,v_0)_V + \lambda^2 \norm{v_1}^2_{H^2(\Omega)}
\\ &= (1-\lambda)^2 \mathcal{J}(A_0)^2 + \lambda^2 \mathcal{J}(A_1)^2 \qquad \forall \lambda \in[0,1],
    \end{split}
\end{equation}
where, in passing to the last line above, we have used the orthogonality $(v_1,v_0)_V=0$ shown above, as well as the identity~\eqref{eq:PDI_VI_equivalence_10}.
Thus~\eqref{eq:PDI_VI_equivalence_8} shows that the function $\lambda\mapsto \mathcal{J}(A_\lambda)^2$ is a quadratic polynomial in $\lambda$.
Since $A_0\in \mathcal{D}H[u]$ is a minimizer of the functional~$\mathcal{J}$ by~\eqref{eq:PDI_VI_equivalence_2}, it is then clear that
\begin{equation}\label{eq:PDI_VI_equivalence_9}
0 \leq \left.\frac{\mathrm{d}}{\mathrm{d}\lambda}\right|_{\lambda =0} \mathcal{J}(A_\lambda)^2 = - 2 \mathcal{J}(A_0)^2,
\end{equation}
and thus we conclude that $\mathcal{J}(A_0)=0$.
As explained above, this shows that $(u,m)$ is a solution of the PDI problem~\eqref{eq:MFG_PDI}.
\end{proof}

\begin{remark}
\label{remark: minimax theorems}
We have aimed to keep the techniques of the proof of Theorem~\ref{thm:PDI_VI_equivalence} as elementary as possible for the sake of simplicity and accessibility.
However, at a more fundamental level, the essence of the proof of Theorem~\ref{thm:PDI_VI_equivalence} can be seen as a verification of the minimax principle
\begin{equation}\label{eq:minimax}
\min_{A\in\mathcal{D}H[u]}\max_{v\in B_V}\mathcal{L}(A,v) = \max_{v\in B_V}\min_{A\in \mathcal{D}H[u]}\mathcal{L}(A,v),
\end{equation}
for the Lagrangian functional~$\mathcal{L}$ defined in~\eqref{eq:PDI_VI_equivalence_0} above.
Observe that it is sufficient to establish~\eqref{eq:minimax} to conclude that a VI solution $(u,m)$ is a PDI solution, since the left-hand side of~\eqref{eq:minimax} is necessarily nonnegative, and the right-hand side is nonpositive by Lemma~\ref{lem:VI_directional_selection}, and thus both sides must be identically zero.
Thus, Theorem~\ref{thm:PDI_VI_equivalence} can also be deduced from well-known minimax theorems, see e.g.~\cite{EkelandTemam1999,Fan1952,Sion1958,Zeidler1986}.
Given that minimax theorems are available in many functional settings under very general structural assumptions, it appears likely that the equivalence between PDI and VI formulations of MFG systems can be generalized to a wider range of problems.
\end{remark}

\section{Elliptic operators with Cordes coefficients}\label{sec:HJB_KFP_cordes}

It is well-known that for a bounded convex domain~$\Omega$, solutions of Poisson's equation on~$\Omega$ with right-hand side in~$L^2(\Omega)$, coupled with a homogeneous Dirichlet boundary condition on~$\partial \Omega$, belong to the space $V\coloneqq H^2(\Omega)\cap H^1_0(\Omega)$, see for instance~\cite[Theorem~3.2.1.2, p.~147]{Grisvard2011}.
Furthermore, in \cite[Theorem~10, p.~391]{Kadlec1964}, Kadlec gave an exact constant for the bound on the solution, in particular,
\begin{equation}\label{eq:Kadlec}
|v|_{H^2(\Omega)} \leq \|\Delta v\|_{L^2(\Omega)} \quad\forall v\in V, 
\end{equation}
where the $H^2$-seminorm is defined by $\abs{v}_{H^2(\Omega)}^2\coloneqq \int_\Omega |D^2v|^2\mathrm{d}x$, with $|D^2v|$ denoting the Frobenius norm of the Hessian $D^2v$.
Indeed, the bound~\eqref{eq:Kadlec} is deduced from~\cite[Theorem~10, p.~391]{Kadlec1964} where, in the notation used there, the constant $\alpha$ appearing there simplifies to $\alpha=1$ in the case of the Laplacian.
Using the Poincar\'e inequality, it is then straightforward to verify that
\begin{equation}\label{eq:Kadlec_2}
\norm{v}_{H^2(\Omega)} \leq C\|\Delta v\|_{L^2(\Omega)} \quad\forall v\in V,
\end{equation}
where the constant $C$ depends only on $d$ and $\diam(\Omega)$.

\begin{remark}[Kadlec--Miranda--Talenti inequality]\label{rem:kadlec}
The proof of~\eqref{eq:Kadlec} can also be found in the book of Grisvard~\cite{Grisvard2011}, and, for more regular domains, in the book~\cite{MaugeriPalagachevSoftova2000}.
Indeed, in~\cite{Grisvard2011}, which follows the approach of Miranda~\cite{Miranda1963} and Talenti~\cite{Talenti1965}, the inequality is shown for a convex domain with sufficiently (piecewise) regular boundary, where the main idea is to integrate by parts the mixed second-order derivatives and use the convexity of the domain to determine the sign of the boundary terms, which are related to the curvature of the boundary, see~\cite[Eq.~(3,1,2,2) p.~139]{Grisvard2011} and the equation immediately preceding it.
The inequality is then later extended to general convex domains by an inner approximation of the domain, see~\cite[Section~3.2.1]{Grisvard2011}.
Note also that Talenti showed that~\eqref{eq:Kadlec} still holds with the same constant if the domain has sufficiently regular boundary and that the mean curvature of the boundary has a certain sign; this is seen in~\cite[Eq.~(5)]{Talenti1965}, where, in the notation of the paper, one takes $f=\frac{1}{N}\Delta u$ with $N$ the spatial dimension, which corresponds to the operator with coefficients $a_{ij}=\frac{1}{N}\delta_{ij}$ the Kronecker delta that satisfies the hypothesis~\cite[Eq.~(2)]{Talenti1965} with $\epsilon=1$, in which case \cite[Eq.~(5)]{Talenti1965} simplifies to~\eqref{eq:Kadlec}.
Talenti further remarked on~\cite[p.~288]{Talenti1965} that for $d=2$ the mean curvature condition is equivalent to convexity, whereas in $d\geq 3$ it allows for some nonconvex domains where the inward curvature in some directions is sufficiently dominant.
As a remark, in the earlier papers~\cite{SmearsSuli2013,SmearsSuli2014}, the inequality~\eqref{eq:Kadlec} was named after Miranda and Talenti, following the reference~\cite{MaugeriPalagachevSoftova2000}, as the second author was not previously aware of~\cite{Kadlec1964}. 
\end{remark}

\subsection{Fully nonlinear HJB equations}

In this section, we recall some known results from \cite{SmearsSuli2013,SmearsSuli2014} on the well-posedness of strong solutions in $V$ of fully nonlinear HJB equations with Cordes coefficients.
Note that well-posedness results of a similar nature are also available for time-dependent HJB equations~\cite{SmearsSuli2016}, and also Isaacs equations~\cite{KaweckiSmears2021,KaweckiSmears2022}.
We assume below that $\Omega$ is a bounded, convex open subset of $\R^d$, and that Assumption~\ref{assumption: Cordes} holds.
Recall that the Hamiltonian defines an operator $V\ni v\mapsto H[v]\in L^2(\Omega)$ defined in~\eqref{eq:H_op_def} above.
In this section, we consider the problem of finding a strong solution $u\in V$ of a fully nonlinear HJB equation of the form
\begin{equation}\label{eqn: HJB}
H[u] = \widetilde{F} \quad\text{in }\Omega,
\end{equation}
where $\widetilde{F}\in L^2(\Omega)$ is a given function. Observe that the homogeneous Dirichlet boundary condition $u=0$ on $\partial \Omega$ is encoded in the choice of space $V$ defined in~\eqref{eq:V_def} above.
The following result is a modest extension of \cite[Theorem~3]{SmearsSuli2014}.
\begin{lemma}\label{lemma: HJB existence}
Let $\Omega$ be a nonempty bounded open convex subset of~$\R^d$, $d\geq 1$.
Let $\mathcal{A}$ be a compact metric space, let $a\in C(\overline{\Omega}\times\mathcal{A};\Rddsymp)$ satisfy Assumption~\ref{assumption: Cordes}, and let $f\in C(\overline{\Omega}\times\mathcal{A})$.
Then, for any $\widetilde{F}\in L^2(\Omega)$, there exists a unique~$u \in V$ solving~\eqref{eqn: HJB} pointwise a.e.\ in~$\Omega$.
There exists a constant $C$, depending only on $d$, $\diam(\Omega)$, $\varepsilon$, $\underline{\nu}$, $\overline{\nu}$, and on $\norm{f}_{C(\overline{\Omega}\times\mathcal{A})}$, such that
\begin{equation}\label{eq:HJB_boundedness}
\norm{u}_{H^2(\Omega)}\leq C\left(1+\|\widetilde{F}\|_{L^2(\Omega)}\right).
\end{equation}
Furthermore, there exists a constant $C$, depending only on $d$, $\diam(\Omega)$, $\varepsilon$, $\underline{\nu}$ and $\overline{\nu}$, such that
\begin{equation} \label{eqn: HJB stability}
\|u_1 - u_2 \|_{H^2(\Omega)} \leq C \|\widetilde{F}_1 - \widetilde{F}_2\|_{L^2(\Omega)} \quad \forall \widetilde{F}_1,\,\widetilde{F}_2\in L^2(\Omega),
\end{equation}
where $u_i$ denotes the unique solution in~$V$ of $H[u_i]=\widetilde{F}_i$ in~$\Omega$, for each $i\in\{1,2\}$.
\end{lemma}

\begin{proof}
The proof essentially follows~\cite[Section~3]{SmearsSuli2014}, so we outline the main ideas of the proof for completeness.
The principal idea is to apply the Browder--Minty theorem for the well-posedness of a strongly monotone operator equation that is obtained by renormalization of the nonlinear Hamiltonian and testing against the Laplacians of test functions.
We define the function $\gamma\in C(\overline{\Omega}\times \mathcal{A})$ by $\gamma(x,\alpha)\coloneqq \frac{\Trace a(x,\alpha)}{\abs{a(x,\alpha)}^2}$ for all $(x,\alpha)\in\overline{\Omega}\times\mathcal{A}$.
Note that continuity of~$\gamma$ follows from uniform ellipticity~\eqref{eq:uniform_ellipticity}, and moreover that there exists a $\gamma_0>0$, depending only on $\underline{\nu}$ and $\overline{\nu}$, such that $\gamma\geq \gamma_0$ on $\overline{\Omega}\times\mathcal{A}$.
It is furthermore easy to verify that $\norm{\gamma}_{C(\overline{\Omega}\times\mathcal{A})}$ can be bounded solely in terms of the constants $\underline{\nu}$ and $\overline{\nu}$. 
Let $\widetilde{F}\in L^2(\Omega)$ be fixed but arbitrary.
We now define the operator $V\ni v\mapsto H_{\gamma,\widetilde{F}}[v]\in L^2(\Omega)$ by
\begin{equation}
H_{\gamma,\widetilde{F}}[v](x)\coloneqq \sup_{\alpha\in\mathcal{A}}\big\{\gamma(x,\alpha)\big(-a(x,\alpha):D^2v(x) - f(x,\alpha) - \widetilde{F}(x)\big)\big\} \quad \text{for a.e. } x\in \Omega.
\end{equation}
Observe that the operator is well-defined, i.e.\ $H_{\gamma,\widetilde{F}}[v]$ is Lebesgue measurable and belongs to $L^2(\Omega)$ for all $v\in V$ and all $\widetilde{F}\in L^2(\Omega)$, since the map $\overline{\Omega}\times \Rddsym\times \R\ni (x,M,z)\mapsto \sup_{\alpha\in\mathcal{A}}\left\{\gamma(x,\alpha)\left(-a(x,\alpha):M - f(x,\alpha) -z\right)\right\} \in \R $ is continuous in $x$ and Lipschitz continuous in $(M,z)$, and the functions $\gamma$, $a$ and $f$ are uniformly bounded over $\overline{\Omega}\times \mathcal{A}$.
Since $\mathcal{A}$ is compact, since the functions $\gamma$, $a$ and $f$ are continuous, and since $\gamma$ is bounded from above and also from below away from zero, i.e.\ $\gamma\geq \gamma_0>0$ on $\overline{\Omega}\times \mathcal{A}$ for some constant $\gamma_0$, it follows that a function $u\in V$ solves $H[u]=\widetilde{F}$ pointwise a.e.\ in $\Omega $ if and only if $H_{\gamma,\widetilde{F}}[u]=0$ pointwise a.e.\ in $\Omega$. 
The argument to show this is from~\cite{SmearsSuli2014}, which we briefly recall here: for a.e.\ $x\in \Omega$, it is clear that $H[u](x)=\widetilde{F}(x)$ if, and only if, $-a(x,\alpha):D^2u(x) - f(x,\alpha)-\widetilde{F}(x)\leq 0$ for all $\alpha\in \mathcal{A}$, with equality achieved by some $\alpha$. This is equivalent to $\gamma(x,\alpha)\left(-a(x,\alpha):D^2u(x) - f(x,\alpha)-\widetilde{F}(x)\right)\leq 0$ for all $\alpha$, again with equality achieved by some $\alpha$, which is equivalent to $H_{\gamma,\widetilde{F}}[u](x)=0$.

Furthermore, since the Laplacian $\Delta\colon V\to L^2(\Omega)$ is bijective for convex $\Omega$, it is clear that $H_{\gamma,\widetilde{F}}[u]=0$ pointwise a.e.\ in $\Omega$ if and only if
\begin{equation}\label{eq:HJB_existence_1}
    \pair{\mathcal{H}_{\gamma,\widetilde{F}}[u]}{v}_{V^*\times V} = 0 \quad \forall v\in V,
\end{equation}
where the operator $\mathcal{H}_{\gamma,\widetilde{F}}\colon V\to V^*$ is defined by $\pair{\mathcal{H}_{\gamma,\widetilde{F}}[u]}{v}_{V^*\times V}\coloneqq (H_{\gamma,\widetilde{F}}[u],-\Delta v)_{\Omega} $ for all $v\in V$.
Thus the problems~\eqref{eqn: HJB} and~\eqref{eq:HJB_existence_1} are equivalent.

We now show that $\mathcal{H}_{\gamma,\widetilde{F}}$ is strongly monotone and Lipschitz continuous on $V$, so the existence and uniqueness of the solution of~\eqref{eq:HJB_existence_1} follows from the Browder--Minty theorem~\cite[Theorem~10.49, p.~364]{RenardyRogers2004}.
Under the Cordes condition~\eqref{eq:Cordes_condition}, we follow the arguments in~\cite[Lemma~1]{SmearsSuli2014} to find that
\begin{equation}\label{eq:HJB_existence_2}
\begin{split}
|H_{\gamma,\widetilde{F}}[w](x)-H_{\gamma,\widetilde{F}}[v](x)+\Delta(w-v)(x)| &\leq\sup_{\alpha\in\mathcal{A}}\abs{\left(I_d-\gamma(x,\alpha)a(x,\alpha)\right):D^2(w-v)(x)} 
\\ &\leq \sqrt{1-\varepsilon}\,\abs{D^2(w-v)(x)},
\end{split}
\end{equation}
for a.e.\ $x\in \Omega$, and all $w$, $v\in V$, where $I_d\in \Rdd$ denotes the identity matrix.
Therefore, we deduce from~\eqref{eq:Kadlec} and~\eqref{eq:HJB_existence_2} that $\mathcal{H}_{\gamma,\widetilde{F}}$ is strongly monotone on $V$ with
\begin{multline}\label{eq:HJB_existence_3}
\pair{\mathcal{H}_{\gamma,\widetilde{F}}[w] - \mathcal{H}_{\gamma,\widetilde{F}}[v]}{w-v}_{V^*\times V} 
\\ =\norm{\Delta(w-v)}_{L^2(\Omega)}^2-(H_{\gamma,\widetilde{F}}[w]-H_{\gamma,\widetilde{F}}[v]+\Delta(w-v),\Delta(w-v))_\Omega
\\  \geq \left(1-\sqrt{1-\varepsilon}\right)\norm{\Delta(w-v)}^2_{L^2(\Omega)}\geq \frac{1-\sqrt{1-\varepsilon}}{C^2}\norm{w-v}_{H^2(\Omega)}^2,
\end{multline}
where $C$ is the constant appearing in~\eqref{eq:Kadlec_2}.
It is also straightforward to check that $\mathcal{H}_{\gamma,\widetilde{F}}$ is Lipschitz continuous from $V$ to $V^*$.
Therefore, we may apply the Browder--Minty theorem to find that there exists a unique $u\in V$ that solves~\eqref{eq:HJB_existence_1}, and thus, equivalently, that solves~\eqref{eqn: HJB}.
The bound~\eqref{eq:HJB_boundedness} is easily found from 
\begin{equation}\label{eq:HJB_existence_5}
    \frac{1-\sqrt{1-\varepsilon}}{C^2}\norm{u}_{H^2(\Omega)}^2\leq - \pair{\mathcal{H}_{\gamma,\widetilde{F}}[0]}{u}_{V^*\times V}\leq C\big\lVert  H_{\gamma,\widetilde{F}}[0]\big\rVert_{L^2(\Omega)}\norm{u}_{H^2(\Omega)},
\end{equation}
which is a consequence of~\eqref{eq:HJB_existence_3}, and from the bound $|H_{\gamma,\widetilde{F}}[0](x)|\leq \norm{\gamma}_{C(\overline{\Omega}\times\mathcal{A})} (\norm{f}_{C(\overline{\Omega}\times\mathcal{A})}+|\widetilde{F}(x)|) $ for a.e.\ $x\in \Omega$. 
We now prove the Lipschitz continuity of the inverse map~\eqref{eqn: HJB stability}. 
For each $i\in\{1,2\}$, let $\widetilde{F}_i \in L^2(\Omega)$ be given, and let $u_i\in V$ denote the corresponding unique solution of $H[u_i]=\widetilde{F}_i$ in $\Omega$, and, equivalently, of $\mathcal{H}_{\gamma,\widetilde{F}_i}[u_i]=0$ in $V^*$.
We use the strong monotonicity~\eqref{eq:HJB_existence_3} to find that
\begin{equation}\label{eq:HJB_existence_4}
\begin{split}
\frac{1}{C}\norm{u_1-u_2}_{H^2(\Omega)}^2
& \leq  \pair{\mathcal{H}_{\gamma,\widetilde{F}_2}[u_1]-\mathcal{H}_{\gamma,\widetilde{F}_2}[u_2]}{ u_1-u_2 }_{V^*\times V} 
\\ & = \pair{\mathcal{H}_{\gamma,\widetilde{F}_2}[u_1]-\mathcal{H}_{\gamma,\widetilde{F}_1}[u_1]}{u_1-u_2}_{V^*\times V},
\end{split}
\end{equation}
where the constant $C$ depends only on $d$, $\diam(\Omega)$ and on $\varepsilon$, and where, in passing to the second equality above, we have used the fact that $\mathcal{H}_{\gamma,\widetilde{F}_2}[u_2]=0=\mathcal{H}_{\gamma,\widetilde{F}_1}[u_1]$. The inequality $|\sup_{\alpha} x^\alpha-\sup_{\alpha}y^\alpha|\leq\sup_{\alpha}|x^\alpha-y^\alpha|$, for arbitrary bounded sets of real numbers $\{x^\alpha\}_{\alpha \in \mathcal{A}}, \{y^\alpha\}_{\alpha \in \mathcal{A}} \subset \mbb{R}$, implies that $|H_{\gamma,\widetilde{F}_2}[u_1](x)-H_{\gamma,\widetilde{F}_1}[u_1](x)|\leq \sup_{\alpha\in\mathcal{A}}|\gamma(x,\alpha)(\widetilde{F}_2(x)-\widetilde{F}_1(x))|$ for a.e.\ $x\in \Omega$. 
Therefore $|H_{\gamma,\widetilde{F}_2}[u_1]-H_{\gamma,\widetilde{F}_1}[u_1]|\leq C|\widetilde{F}_1-\widetilde{F}_2|$ a.e.\ in $\Omega$, for a constant $C$ depending only on $\underline{\nu}$ and $\overline{\nu}$.
Therefore, we obtain~\eqref{eqn: HJB stability} by applying the Cauchy--Schwarz inequality to~\eqref{eq:HJB_existence_4}, and using the above bounds to simplify.
\end{proof}

\begin{remark}[Convexity of~$\Omega$]\label{rem:convexity_omega}
The convexity assumption on $\Omega$ plays two roles in the analysis.
First, it justifies the $H^2$-regularity, up to the boundary, of solutions of the Dirichlet problem for Poisson's equation with right-hand side in $L^2(\Omega)$, thereby justifying the equivalence between the HJB equation and the variational problem~\eqref{eq:HJB_existence_1}.
Second, it is used to determine the exact constant in the Kadlec--Miranda--Talenti inequality~\eqref{eq:Kadlec}, which is then used in conjunction with the Cordes condition in the proof of the strong monotonicity bound~\eqref{eq:HJB_existence_3}.
This convexity assumption can be removed to a somewhat limited extent.
As already mentioned above in Remark~\ref{rem:kadlec}, Talenti~\cite{Talenti1965} showed that the inequality~\eqref{eq:Kadlec} generalizes to domains with a sufficiently regular boundary whose mean curvature has a certain sign.
Alternatively, if one assumes that the boundary~$\partial\Omega$ is $C^{1,1}$ regular without condition on its curvature, one can further assume that the constant~$\varepsilon$ in the Cordes condition is additionally sufficiently close to $1$, in order to compensate for the fact that generalization of the inequality in~\eqref{eq:Kadlec}, for a nonconvex $C^{1,1}$ domain, holds up to some generic constant. 
However, given that it may be difficult to determine these constants in practice, we do not pursue this direction further.
\end{remark}

\subsection{Kolmogorov--Fokker--Planck equations}

We now turn to KFP equations with discontinuous coefficients of the general form: find $m \in L^2(\Omega)$ such that
\begin{equation}\label{eqn: weak KFP}
    \begin{aligned}
    (m,-\overline{a}:D^2v)_\Omega=\pair{G}{v}_{V^*\times V} &&& \forall v\in V,
    \end{aligned}
\end{equation}
where $G\in V^*$, and $\overline{a}\in L^\infty(\Omega;\Rddsymp)$ are given as data, satisfying the condition $\overline{a}(x)\in\Cclass$ for a.e.\ $x\in \Omega$, where it is recalled that $\Cclass$ is the subset of matrices in $\Rddsymp$ that satisfy the uniform ellipticity condition~\eqref{eq:uniform_ellipticity} and the Cordes condition~\eqref{eq:Cordes_condition}.
Recall that~\eqref{eqn: weak KFP} corresponds to the very weak formulation of the KFP equation $-D^2{:}(\overline{a} m)=G$ in $\Omega$ with $m=0$ on $\partial\Omega$, with $D^2{:}$ denoting the double divergence operator, after integration-by-parts.

We will use the following theorem due to~Ne\v{c}as in the proof of the existence and uniqueness of the solution of~\eqref{eqn: weak KFP}.

\begin{theorem}[{\cite[Theorem 3.1]{Necas1962}}]
\label{thm:necas}
    Let $V_1$ and $V_2$ be two Hilbert spaces, and $B: V_1 \times V_2 \rightarrow \mbb{R}$ a bilinear form such that
    \begin{subequations}
    \begin{alignat}{2}
        |B(w,v)| &\leq C_1 \|w\|_{V_1} \|v\|_{V_2}  &\quad& \forall (w,v)\in V_1\times V_2,\label{eqn:necas1}\\
        C_2 \|w\|_{V_1} & \leq \sup_{ v \in B_{V_2} } B(w,v)   &\quad& \forall w\in V_1, \,\label{eqn:necas2}\\
        C_3 \|v\|_{V_2} &\leq \sup_{ w \in B_{V_1} } B(w,v)   &\quad& \forall v\in V_2,\label{eqn:necas3}
    \end{alignat}
    \end{subequations}
    for some constants $C_1, C_2, C_3 > 0$, where $B_{V_1}$ and $B_{V_2}$ denote the closed unit balls in $V_1$ and $V_2$, respectively.
    Then, for any linear functional $\ell \in V_2^*$, there exists a unique element $w \in V_1$ such that $B(w,v) = \ell(v)$ for all  $v \in V_2$.
\end{theorem}

\begin{remark}
The conditions~\eqref{eqn:necas2} and~\eqref{eqn:necas3} in Theorem~\ref{thm:necas}, and its generalizations to reflexive Banach spaces, are commonly called the inf-sup conditions in the numerical analysis literature, since they are equivalent to
\[
\inf_{w\in V_1\setminus\{0\}} \sup_{v \in V_2\setminus\{0\}} \frac{B(w,v)}{\norm{w}_{V_1}\|v\|_{V_2}}\geq C_2, \quad \inf_{v\in V_2\setminus\{0\}}  \sup_{w \in V_1\setminus\{0\}} \frac{B(w,v)}{\|w\|_{V_1}\norm{v}_{V_2}} \geq C_3.
\]
We refer the reader to~\cite[Theorem~2.6]{ErnGuermond2004}, which is named the Banach--Ne\v{c}as--Babu\v{s}ka Theorem, for the generalization of Theorem~\ref{thm:necas} to the case of a bilinear form posed on a pair of Banach spaces, with the second one assumed reflexive, which gives both necessary and sufficient conditions for the well-posedness of linear problems involving bilinear forms.
\end{remark}

We now show the well-posedness of the KFP equation~\eqref{eqn: weak KFP}.

\begin{lemma}
\label{lem: KFP well posed}
Let $\Omega$ be a nonempty bounded open convex subset of $\R^d$, $d\geq 1$.
Let $\overline{a}\in L^\infty(\Omega;\Rddsymp)$ satisfy $\overline{a}(x)\in\Cclass$ for a.e.\ $x\in \Omega$.
Then, for each $G\in V^*$, there exists a unique $m \in L^2(\Omega)$ that solves~\eqref{eqn: weak KFP}.
Moreover, there exists a constant $C$, depending only on $d$, $\diam(\Omega)$, $\varepsilon$, $\underline{\nu}$ and $\overline{\nu}$, such that
\begin{equation} \label{eqn: KFP L2 control}
     \|m\|_{L^2(\Omega)} \leq C\|G\|_{V^*}.
\end{equation}
\end{lemma}
\begin{proof}
We consider the pair of Hilbert spaces $L^2(\Omega)$ and $V$ and the bilinear form $B\colon L^2(\Omega)\times V\to \R$ defined by
\[B(m,v)\coloneqq (m,-\overline{a}:D^2v)_\Omega \quad \forall (m,v)\in L^2(\Omega)\times V.\]
It is clear that $B$ is bounded on $L^2(\Omega)\times V$, i.e.~$B$ satisfies the condition~\eqref{eqn:necas1} for some constant with $V_1=L^2(\Omega)$ and $V_2=V$, since $\overline{a}$ is essentially bounded on $\Omega$. 
By~\cite[Theorem~3]{SmearsSuli2013}, see also~\cite[Theorem~1.2.1]{MaugeriPalagachevSoftova2000}, for each $g\in L^2(\Omega)$ there exists a unique $v_g\in V$ such that $-\overline{a}:D^2v_g=g$ a.e.\ in $\Omega$, and furthermore $\norm{v_g}_{H^2(\Omega)}\leq C\norm{g}_{L^2(\Omega)}$ for a constant $C>0$ depending only on $d$, $\diam(\Omega)$, $\varepsilon$, $\underline{\nu}$ and $\overline{\nu}$.
We now show that $B$ satisfies~\eqref{eqn:necas2} and \eqref{eqn:necas3}, in particular that
\begin{subequations}\label{eqn: infsup}
\begin{alignat}{2}
 \frac{1}{C}\norm{m}_{L^2(\Omega)} &\leq \sup_{ v\in B_V} B(m,v) &\quad& \forall m\in L^2(\Omega), \label{eq:infsup_1}
 \\
 \frac{1}{C}\norm{v}_{H^2(\Omega)} &\leq \sup_{ m\in B_{L^2(\Omega)} } B(m,v) &\quad& \forall v\in V,\label{eq:infsup_2}
\end{alignat}
\end{subequations}
with the same constant~$C$ above, where~$B_V$ and~$B_{L^2(\Omega)}$ denote the closed unit balls in~$V$ and~$L^2(\Omega)$, respectively.
Indeed, since \eqref{eq:infsup_1} and~\eqref{eq:infsup_2} hold trivially for the cases $m=0$ and $v=0$, it is enough to consider only $m\neq 0$ and $v\neq 0$.
Then~\eqref{eq:infsup_1} follows from the choice $g=m$, in particular from $\norm{m}_{L^2(\Omega)}^2=B(m,v_m)$ and $\norm{v_m}_{V}\leq C\norm{m}_{L^2(\Omega)}$.
Similarly,~\eqref{eq:infsup_2} follows from the choice $g=-\overline{a}:D^2v$.
We may therefore apply Theorem~\ref{thm:necas} to conclude that, for any $G\in V^*$, there exists a unique $m\in L^2(\Omega)$ that solves $B(m,v)=\pair{G}{v}_{V^*\times V}$ for all $v\in V$, i.e.\ $m$ solves~\eqref{eqn: weak KFP}. The bound~\eqref{eqn: KFP L2 control} follows immediately from~\eqref{eq:infsup_1}.
\end{proof}

\subsection{Comparison principles for operators with Cordes coefficients}\label{sec:comparison}

The following result, which appears to be original, gives a comparison principle for nondivergence form elliptic equations with discontinuous coefficients that satisfy the Cordes condition, on general bounded convex domains with Dirichlet boundary conditions.

\begin{theorem}\label{thm:comparison_nondivergence}
Let $\Omega\subset \R^d$ be an open, bounded, convex set.
Let $\overline{a} \in L^\infty(\Omega;\Rddsymp)$ satisfy $\overline{a}(x)\in \Cclass$ for a.e.\ $x\in\Omega$. 
Let $ g \in L^2(\Omega) $ and let $ v_g \in V $ be the unique function that satisfies $-\overline{a}:D^2 v_g=g$ pointwise a.e.\ in $\Omega$, and $v_g=0$ on $\partial \Omega$ in the trace sense.
If $g$ is nonnegative a.e.\ in $\Omega$, then $v_g$ is nonnegative a.e.\ in~$\Omega$.
\end{theorem}

The function $v_g\in V$ appearing in the statement of Theorem~\ref{thm:comparison_nondivergence} is called a strong solution of the equation $-\overline{a}:D^2 v_g=g$ in $\Omega$ and $v_g=0$ on $\partial \Omega$, and its existence and uniqueness is well-known, see \cite{MaugeriPalagachevSoftova2000,SmearsSuli2013}.
Recall that the homogeneous Dirichlet boundary condition is incorporated into the space $V=H^2(\Omega)\cap H^1_0(\Omega)$.
The original aspect of Theorem~\ref{thm:comparison_nondivergence} is the conclusion of nonnegativity of the solution for nonnegative data.
Since the proof of Theorem~\ref{thm:comparison_nondivergence} is rather technical, we postpone it to Appendix~\ref{appendix: comparison principle}.

\begin{remark}
It is known that the conclusion of Theorem~\ref{thm:comparison_nondivergence} generally fails if one drops the hypothesis that the coefficient~$\overline{a}$ satisfies the Cordes condition~\eqref{eq:Cordes_condition}, whenever $d\geq 3$. 
Indeed, a family of examples is given in~\cite[p.~18]{MaugeriPalagachevSoftova2000} of linear uniformly elliptic equations on the unit ball $B_1$ in $\R^d$, $d\geq 3$ arbitrary, with discontinuous coefficients that do not satisfy the Cordes condition, such that, in each case, there exist two different strong solutions in $H^2(B_1)\cap H^1_0(B_1)$. By linearity, these give counterexamples for any comparison principle.
Moreover, it is shown in~\cite[p.~20]{MaugeriPalagachevSoftova2000} that the Cordes condition is essentially sharp for these examples, since the coefficients are discontinuous at a single point, and come arbitrarily close to satisfying the Cordes condition.
This explains the need for the additional assumption of the Cordes condition in Theorem~\ref{thm:comparison_nondivergence}. 
Note that this stands in contrast to the Aleksandrov--Bakelman--Pucci maximum principle, which only requires uniform ellipticity, but only applies to functions in $W^{2,d}_{\mathrm{loc}}(\Omega)$, i.e.\ functions with higher local integrability of the second-order derivatives when $d\geq 3$.
\end{remark}

The principal use of Theorem~\ref{thm:comparison_nondivergence} for our purposes is that it entails the following comparison principle for very weak solutions of KFP equations related to the MFG system.

\begin{corollary}\label{cor:kfp_comparison}
Let $\Omega$ be a nonempty bounded open convex subset of $\R^d$, $d\geq 1$.
Let $\overline{a} \in L^\infty(\Omega;\Rddsymp)$ satisfy $\overline{a}(x)\in \Cclass$ for almost every $x\in\Omega$. 
Let $G\in V^*$ and let $m\in L^2(\Omega)$ denote the unique solution of~\eqref{eqn: weak KFP}.
If $G$ is nonnegative, then $m$ is nonnegative a.e.\ in $\Omega$.
\end{corollary}

\begin{proof}
The existence and uniqueness of the solution $m$ is shown in Lemma~\ref{lem: KFP well posed}.
In order to show that $m$ is nonnegative a.e.\ in $\Omega$, we will prove that the negative part $m^-=0$ a.e.\ in $\Omega$, where $m^- \coloneqq\max\{-m,0\}$ denotes the negative part of $m$.
Let $\chi_{\{m<0\}}\in L^\infty(\Omega)$ denote the indicator function of the set $\{x\in \Omega : m(x) < 0\}$, and let $v_*\in V$ be the unique strong solution of $-\overline{a}:D^2v_* = \chi_{\{m<0\}}$ in $\Omega$ with $v_*=0$ on $\partial \Omega$.
Theorem~\ref{thm:comparison_nondivergence} implies that $v_*\geq 0$ a.e.\ in $\Omega$. 
Then, using the test function $v_*$ in~\eqref{eqn: weak KFP} leads to
\begin{equation*}
0\leq \pair{G}{v_*}_{V^* \times V} = \left( m, -\overline{a}:D^2 v_*\right)_\Omega = \int_{\{m<0\}} m \, \dx = - \norm{m^-}_{L^1(\Omega)}.
\end{equation*}
Therefore $\norm{m^-}_{L^1(\Omega)}\leq 0$ and we conclude that~$m$ is nonnegative a.e.\ in~$\Omega$.
\end{proof}

We now prove Corollary~\ref{cor:PDI_VI_equivalence_complete}.

\begin{proof}[Proof of Corollary~\ref{cor:PDI_VI_equivalence_complete}]
By Theorem~\ref{thm:PDI_VI_equivalence}, it is enough to prove that if $(u,m)\in V\times L^2(\Omega)$ is a solution of the PDI problem~\eqref{eq:MFG_PDI}, then $m\geq 0$ a.e.\ in $\Omega$.
By definition, there exists an $-\overline{a}\in\mathcal{D}H[u]$ such that $(m,-\overline{a}:D^2v)_\Omega=\pair{G}{v}_{V^*\times V}$ for all $v\in V$.
By Lemma~\ref{lem:Cordes_subdifferential}, we have $\overline{a}(x)\in\Cclass$ for a.e.\ $x\in \Omega$. Therefore, using the hypothesis that $G\in V^*$ is nonnegative, we may apply Corollary~\ref{cor:kfp_comparison} to conclude that $m\geq 0$ a.e.\ in $\Omega$.
\end{proof}

\section{Existence and uniqueness of solutions of fully nonlinear MFG PDI with Cordes coefficients}
\label{section: well-posed}

\subsection{Proof of the existence of a solution}

In this section, we place ourselves in the setting of the hypotheses of Theorem~\ref{thm: MFG existence}.
We will show the existence of solutions of the VI problem, and equivalently of the PDI problem, by applying the Kakutani fixed point theorem to a suitably defined set-valued map, which we now construct in several steps.
Recall that Lemma~\ref{lemma: HJB existence} shows that, for each $m\in L^2(\Omega)$, there exists a unique strong solution $u\in V$ of the HJB equation $H[u]=F[m]$ in $\Omega$.
This defines the solution map $T\colon L^2(\Omega)\to V$, i.e.\ $T\colon m\mapsto T(m)$ for each $m\in L^2(\Omega)$ with
\begin{equation}\label{eq:Tmap_def}
\begin{aligned}
T(m) &\coloneqq u, &H[u] &= F[m] \quad\text{in }\Omega.
\end{aligned}
\end{equation}
We also know from Lemma~\ref{lemma: HJB existence} that 
\begin{equation}\label{eq:T_continuity}
\begin{aligned}
\norm{T(m_1)-T(m_2)}_{H^2(\Omega)}\leq C \norm{F[m_1]-F[m_2]}_{L^2(\Omega)} &&& \forall m_1,\,m_2 \in L^2(\Omega),
\end{aligned}
\end{equation}
for a constant $C$ independent of $m_1$ and $m_2$.

We now define a set-valued map $S\colon V\rightrightarrows L^2(\Omega)$ as follows.
Recalling the bound~\eqref{eqn: KFP L2 control}, let $\mathcal{B}\coloneqq \{m\in L^2(\Omega): \norm{m}_{L^2(\Omega)}\leq C\norm{G}_{V^*}\}$ denote the closed ball of radius $R=C\norm{G}_{V^*}$ centred at the origin in $L^2(\Omega)$, where the constant $C$ is the same as the one in~\eqref{eqn: KFP L2 control}.
For each $v\in V$, let $S(v)$ denote the set of all $\widetilde{m}\in L^2(\Omega)$ that satisfy
\begin{equation}\label{eq:Sv_set_def}
\widetilde{m}\in \mathcal{B}, \quad \widetilde{m}\geq 0 \text{ a.e.\ in }\Omega, \quad 
\pair{G}{w-v}_{V^* \times V} \leq \left( \widetilde{m},  H[w]-H[v]\right)_\Omega \quad \forall w\in V.
\end{equation}
We will show that $S(v)$ is nonempty for all $v\in V$ as part of Lemma~\ref{lem:Smap_properties} below.
Note that, by definition, $S(v)\subset L^2_+(\Omega)\cap \mathcal{B}$ for all $v\in V$, hence the sets $S(v)$, $v\in V$, are uniformly bounded in $L^2(\Omega)$ by definition.
The following lemma shows several essential properties of the set-valued map $S$.

\begin{lemma}\label{lem:Smap_properties}
Let $\Omega$ be an open, bounded, convex subset of $\R^d$, $d\geq 1$. 
Let $\mathcal{A}$ be a compact metric space, let $a\in C( \overline{\Omega}\times \mathcal{A};\Rddsymp)$ satisfy Assumption~\ref{assumption: Cordes}, and let $f\in C(\overline{\Omega}\times \mathcal{A})$.
Let $G\in V^*$ be nonnegative.
Then, for each $v\in V$, the set $S(v)$ is a nonempty, bounded, closed, convex subset of $L^2(\Omega)$.
\end{lemma}
\begin{proof}
Let $v\in V$ be fixed but arbitrary.
The set $S(v)$ is bounded since $S(v)\subset \mathcal{B}$ by definition.
To show that $S(v)$ is nonempty for any $v\in V$, we recall that $\mathcal{D}H[v]$ is nonempty by Lemma~\ref{lem:DH_nonempty_bounded}, so we choose some $-\overline{a}\in \mathcal{D}H[v]$.
Lemma~\ref{lem:Cordes_subdifferential} shows that $\overline{a}(x)\in \Cclass$ for a.e.\ $x\in \Omega$. 
Therefore, we may apply Lemma~\ref{lem: KFP well posed} to deduce that there exists a unique $\widetilde{m}\in L^2(\Omega)$ that solves $\left( \widetilde{m}, -\overline{a}:D^2 w \right)_\Omega =\pair{G}{w}_{V^*\times V}$ for all $w\in V$.
Furthermore, the bound~\eqref{eqn: KFP L2 control} implies that $\norm{\widetilde{m}}_{L^2(\Omega)}\leq C\norm{G}_{V^*}$ so $\widetilde{m}\in \mathcal{B}$.
It is additionally clear from Corollary~\ref{cor:kfp_comparison} that $\widetilde{m}\geq 0$ a.e.\ in $\Omega$, since $G$ is nonnegative.
Moreover we find that $\pair{G}{w-v}_{V^*\times V}\leq(\widetilde{m},H[w]-H[v])_\Omega$ for all $w\in V$ by following the same argument as in the proof of Lemma~\ref{lem:PDI_solution_are_VI_solution}.
This shows that $\widetilde{m}\in S(v)$ and thus $S(v)$ is nonempty.
The set $S(v)$ is also convex, since, for any $\widetilde{m}_0$ and $\widetilde{m}_1$ in $S(v)$, and any $\lambda\in[0,1]$, then $(1-\lambda)\widetilde{m}_0+\lambda \widetilde{m}_1\in \mathcal{B}\cap L^2_+(\Omega)$, and also, for any test function~$w\in V$,
\begin{multline}
((1-\lambda) \widetilde{m}_0+\lambda \widetilde{m}_1,H[w]-H[v])_\Omega 
 = (1-\lambda)(\widetilde{m}_0,H[w]-H[v])_\Omega + \lambda (\widetilde{m}_1,H[w]-H[v])_\Omega
\\  \geq (1-\lambda)\pair{G}{w-v}_{V^* \times V}+\lambda \pair{G}{w-v}_{V^* \times V} 
= \pair{G}{w-v}_{V^* \times V}.
\end{multline}
Thus $(1-\lambda) \widetilde{m}_0+\lambda \widetilde{m}_1\in S(v)$ for all $\widetilde{m}_0$ and $\widetilde{m}_1$ in $S(v)$ and all $\lambda\in[0,1]$, so $S(v)$ is convex.
It is furthermore clear that $S(v)$ is closed in $L^2(\Omega)$.
\end{proof}

We now define the set-valued map $\Phi\colon  L^2(\Omega)\rightrightarrows L^2(\Omega)$ as
\begin{equation}\label{eq:Phi_fp_map}
\Phi(m) \coloneqq S(T(m)) \quad \forall m\in L^2(\Omega),
\end{equation}
where it is recalled that $T$, respectively $S$, is defined in~\eqref{eq:Tmap_def}, respectively~\eqref{eq:Sv_set_def}, above.
It is clear that a pair $(u,m)\in V\times L^2_+(\Omega)$ is a solution of the VI problem~\eqref{eq:MFG_VI} if and only if $m$ is a fixed point of $\Phi$, i.e.\ $m\in \Phi(m)$, and $u=T(m)$ is the solution of the corresponding HJB equation.
We will thus prove Theorem~\ref{thm: MFG existence} on the existence of solutions of the PDI and VI problems by applying Kakutani's fixed point theorem to the set-valued map $\Phi$ defined above.
For completeness, we recall Kakutani's theorem below, and refer the reader to~\cite[Theorem~9.B, p.~452]{Zeidler1986} for its proof.

\begin{theorem}[Kakutani's fixed point theorem]\label{thm-kakutani}
Let $\mathcal{B}$ be a nonempty, compact, convex set in a locally convex space $\mathcal{X}$.
Let $\Phi\colon\mathcal{B}\rightrightarrows\mathcal{B}$ be a set-valued map such that $\Phi(m)$ is nonempty, closed and convex for all $m\in\mathcal{B}$. Suppose that $\Phi$ is upper semicontinuous.
Then $\Phi$ has a fixed point in~$\mathcal{B}$, i.e.\ there exists an $m\in\mathcal{B}$ such that $m\in \Phi(m)$. 
\end{theorem}

We are now ready to prove Theorem~\ref{thm: MFG existence} on the existence of solutions of the PDI and VI problems.
 
\begin{proof}[Proof of Theorem~\ref{thm: MFG existence}]
We show that the set-valued map $\Phi $ defined in~\eqref{eq:Phi_fp_map} above satisfies all the hypotheses of Theorem~\ref{thm-kakutani}.
Consider the space $L^2(\Omega)$ equipped with the weak topology, which is a locally convex topological vector space.
We take here $\mathcal{B}=\{m\in L^2(\Omega)\colon \norm{m}_{L^2(\Omega)}\leq C\norm{G}_{V^*}\}$, where the constant $C$ is from~\eqref{eqn: KFP L2 control}.
Hence $\mathcal{B}$ is nonempty, convex and compact in the weak topology of $L^2(\Omega)$.
Note that the separability of $L^2(\Omega)$ implies that the weak topology on $\mathcal{B}$ is metrizable.
Recall that the set-valued map $S$ maps elements of $V$ to subsets of $\mathcal{B}$ by definition, see~\eqref{eq:Sv_set_def} above.
Hence $\Phi\colon \mathcal{B}\rightrightarrows \mathcal{B}$.
Furthermore, by Lemma~\ref{lem:Smap_properties} above, for any $m\in \mathcal{B}$, the set $\Phi(m)=S(T(m))\subset \mathcal{B}$ is nonempty, (strongly) closed, and convex.
Thus, by Mazur's theorem, the set~$\Phi(m)$ is also closed in the weak topology of $L^2(\Omega)$, for any $m\in L^2(\Omega)$.

It remains only to show that $\Phi$ is upper-semicontinuous. 
Since $\Phi \colon \mathcal{B}\rightrightarrows \mathcal{B} $ with nonempty, closed, convex images, with $\mathcal{B}$ compact in the weak topology, the upper-semicontinuity of~$\Phi$ is equivalent to proving that $\Phi$ has closed graph, see~\cite[Corollary~1, p.~42]{AubinCellina1984}.
In other words, for any given sequence $\{(m_j,\widetilde{m}_j)\}_{j\in\N}\subset \mathcal{B}\times\mathcal{B}$, such that $\widetilde{m}_j\in \Phi(m_j)$ for all $j\in\N$, and such that $m_j\rightharpoonup m$ and $\widetilde{m}_j\rightharpoonup \widetilde{m}$ in $L^2(\Omega)$ as $j\to \infty$, it is enough to show that $\widetilde{m}\in\Phi(m)$.
Suppose that we are given such a sequence $\{(m_j,\widetilde{m}_j)\}_{j\in\N}$.
Note that $\widetilde{m}_j\in \Phi(m_j)$ implies that $\widetilde{m}_j\geq 0$ a.e.\ in $\Omega$ for all $j\in \N$, and thus Mazur's theorem implies that $\widetilde{m}\geq 0$ a.e.\ in $\Omega$.
It is also clear that $\widetilde{m}\in\mathcal{B}$.
Let us denote $u_j\coloneqq T(m_j)$ for each $j\in \N$ and $u\coloneqq T(m)$.
The complete continuity of~$F$ implies that $F[m_j]\to F[m]$ in $L^2(\Omega)$, and hence also that $H[u_j] \rightarrow H[u]$, in $L^2(\Omega)$ as $j\to \infty$.
It also follows from~\eqref{eq:T_continuity} that $u_j\to u$ strongly in~$V$ as $j\to \infty$. 
Since, by hypothesis, we have $\widetilde{m}_j\in \Phi(m_j)$ for each $j\in\N$, we get
\begin{equation}\label{eq: MFG existence_1}
\pair{G}{v-u_j}_{V^* \times V}\leq \left(\widetilde{m}_j,H[v]-H[u_j]\right)_{\Omega} \quad \forall j\in \N, \quad \forall v\in V.
\end{equation}
It then follows from the weak convergence $\widetilde{m}_j\rightharpoonup \widetilde{m}$ in $L^2(\Omega)$, the strong convergence $H[u_j]\to H[u]$ in $L^2(\Omega)$, and the strong convergence $u_j\to u$ in $V$ as $j\to \infty$, that we can pass to the limit in~\eqref{eq: MFG existence_1} to find that
\begin{equation}\label{thm: MFG existence_2}
\pair{G}{v-u}_{V^* \times V}\leq \left(\widetilde{m},H[v]-H[u]\right)_\Omega \quad\forall v\in V.
\end{equation}
Recalling that $\widetilde{m}\geq 0$ a.e.\ in $\Omega$ was already shown above, we conclude from~\eqref{thm: MFG existence_2} that $\widetilde{m}\in S(u)=\Phi(m)$, and thus $\Phi$ is upper semicontinuous. 
Therefore, Theorem~\ref{thm-kakutani} shows that there exists an $m\in \mathcal{B}$ such that $m\in \Phi(m)$. 
It is then clear that $(u,m)\coloneqq(T(m),m)$ is a solution of the VI problem~\eqref{eq:MFG_VI}.
By Theorem~\ref{thm:PDI_VI_equivalence}, the VI solution $(u,m)$ is equivalently a solution of the PDI problem~\eqref{eq:MFG_PDI}.

We now turn towards the proof of the \emph{a priori} bounds~\eqref{eq:solution_boundedness_u} and~\eqref{eq:solution_boundedness_m} for generic solutions. 
Let $(u,m)\in V\times L^2_+(\Omega)$ be any solution of the VI problem~\eqref{eq:MFG_VI}, and, equivalently, of the PDI problem~\eqref{eq:MFG_PDI}.
By definition, there exists an $\overline{a}\in L^\infty(\Omega;\Rddsymp)$ such that $-\overline{a}\in\mathcal{D}H[u]$ and $(m,-\overline{a}:D^2 v)_\Omega=\pair{G}{v}_{V^*\times V}$ for all $v\in V$.
Since $\overline{a}(x)\in\Cclass$ for a.e.\ $x\in\Omega$ by Lemma~\ref{lem:Cordes_subdifferential}, we may apply Lemma~\ref{lem: KFP well posed} to deduce that $\norm{m}_{L^2(\Omega)}\leq C\norm{G}_{V^*}$, which is~\eqref{eq:solution_boundedness_m}.
Then, the bound~\eqref{eq:HJB_boundedness} of Lemma~\ref{lemma: HJB existence} shows that $\norm{u}_{H^2(\Omega)}\leq C\Bigl(1+\norm{F[m]}_{L^2(\Omega)}\Bigr)$.
Since the operator $F$ is completely continuous, and thus maps bounded sets to bounded sets, we have $\norm{F[m]}_{L^2(\Omega)}\leq \sup_{\widetilde{m}\in \mathcal{B}}\norm{F[\widetilde{m}]}_{L^2(\Omega)}<\infty$.
This implies~\eqref{eq:solution_boundedness_u}.
\end{proof}

\begin{remark}
Although the analysis above for showing the existence of solutions makes greater use of the VI formulation of the problem, the attentive reader may notice that the PDI point of view plays a key role in showing the nonemptiness of the sets~$S(v)$ in the proof of Lemma~\ref{lem:Smap_properties}. 
This is, of course, an essential step in the analysis towards applying the Kakutani fixed point theorem.
It is thus seen that both points of view are complementary.
\end{remark}

\subsection{Proof of the uniqueness of the solution}\label{sec:uniqueness}

Recall that for the uniqueness of the solution, the coupling term $F$ is assumed to satisfy the strict monotonicity condition of Assumption~\ref{assumption: strict monotonicity}.
The proof of Theorem~\ref{thm:MFG_uniqueness} is a straightforward consequence of the results shown above.

\begin{proof}[Proof of Theorem~\ref{thm:MFG_uniqueness}]
Recall that, by Corollary~\ref{cor:PDI_VI_equivalence_complete}, any PDI solution is a VI solution and thus it is enough to consider only the uniqueness of the solution of the VI problem.
Recall that Lemma~\ref{lem:VI_density_uniqueness} shows that, if $(u_i,m_i)\in V\times L^2_+(\Omega)$ are both solutions of the VI problem~\eqref{eq:MFG_VI} for $i\in\{1,2\}$, then $m_1=m_2$. 
Then, the uniqueness of the solution of the corresponding HJB equation, see Lemma~\ref{lemma: HJB existence}, implies that $u_1=u_2$.
\end{proof}

\section{Continuous dependence under regularization and perturbation of the data}
\label{section: regularization}

We now demonstrate the usefulness and flexibility of the equivalence of PDI and VI formulations of the problem in analysing further properties of the MFG system.
More specifically, we consider continuous dependence results with regards to perturbation of the problem data.
In particular, the VI formulation allows for a very effective and simple treatment of limiting arguments in the absence of strong compactness in the player density functions for sequences of solutions of perturbed problems.

\subsection{Regularization of the Hamiltonian}
We show below how the PDI/VI problems for MFG with nondifferentiable Hamiltonians arise naturally as the limit of MFG PDE with differentiable Hamiltonians.
Consider a family of regularized problems of the form: find $(u_\lambda,m_\lambda)\in V\times L^2_+(\Omega)$ that solve
\begin{subequations}
\begin{alignat}{3}
H_\lambda[u_\lambda] &= F[m_\lambda] & \quad & \text{a.e. in } \Omega, \label{eqn: regularized MFG1}
\\ \left(m_\lambda , \frac{\partial H_\lambda}{\partial M}(x,D^2 u_\lambda) \colon D^2 v  \right)_{\Omega} &= \angled{G, v}_{V^* \times V} & \quad & \forall v \in V, \label{eqn: regularized MFG2}
\end{alignat}
\label{eqn: regularized MFG}
\end{subequations}
where $\lambda\in (0,1]$ is a regularization parameter, and where $H_\lambda \colon \overline{\Omega} \times \Rddsym \rightarrow \R$ is a regularized approximation of $H$ with $H_\lambda \to H$ as $\lambda \to 0$, see~\eqref{eq:H_reg_conv} below.
More precisely, we will make the following assumptions on the family of regularized Hamiltonians $\{H_\lambda\}_{\lambda \in (0,1]}$.
\begin{assumption} \label{assumption: regularised hamiltonian}
The family of functions $\{H_\lambda\}_{\lambda \in (0,1]}$ is such that, for each $\lambda \in (0,1]$, the following hold:
\begin{itemize}

     \item there exists a compact metric space $\mathcal{A}_{\lambda}$ and continuous functions $a_\lambda \colon \overline{\Omega}\times \mathcal{A}_\lambda\tends \Rddsymp$ and $f_\lambda\colon \overline{\Omega}\times \mathcal{A}_\lambda\tends \R$ such that 
    \begin{equation}\label{eq:regularized_representation}
        H_\lambda(x,M) = \sup_{\alpha\in \mathcal{A}_{\lambda}}\{-a_\lambda(x,\alpha):M- f_\lambda(x,\alpha)\} \quad \forall(x,M)\in\overline{\Omega}\times \Rddsym,
    \end{equation}
     with $a_\lambda(x,\alpha)\in \Cclass$ for all $(x,\alpha)\in\overline{\Omega}\times\mathcal{A}_\lambda$;
    
    \item for all~$x \in \overline{\Omega}$, the map $\Rddsym \ni M \mapsto H_\lambda(x,M)$ is continuously differentiable, and the partial derivative~$\frac{\partial H_\lambda}{\partial M}$ is continuous on $\overline{\Omega}\times \Rddsym$;
    \item there exists a continuous function~$\omega \colon [0,1] \rightarrow \R_{\geq 0}$ such that~$\omega(0) = 0$ and
    \begin{equation}\label{eq:H_reg_conv}
    |H_\lambda(x,M) - H(x,M)| \leq \omega(\lambda) \quad \forall (x,M) \in \overline{\Omega} \times \Rddsym.
    \end{equation}
\end{itemize}
\end{assumption}

In~Example~\ref{example: moreau yosida} below, we show that Assumption~\ref{assumption: regularised hamiltonian} holds in the case of Moreau--Yosida regularization.
The assumptions above have several implications.
First, it follows from~\eqref{eq:regularized_representation} that $M\mapsto H_\lambda(x,M)$ is convex over $\Rddsym$, for all $x\in \overline{\Omega}$ and all $\lambda\in(0,1]$.
Observe that the condition $a_\lambda(x,\alpha)\in \Cclass$ for all $(x,\alpha)\in \overline{\Omega}\times \mathcal{A}_\lambda$ implies that
the partial derivative $\frac{\partial H_\lambda}{\partial M}$ satisfies
\begin{equation}\label{eq:regularized_cordes}
\begin{aligned}
-\frac{\partial H_\lambda}{\partial M}(x,M) \in \mathcal{C}(\underline{\nu}, \overline{\nu}, \varepsilon) \quad \forall (x,M)\in \overline{\Omega}\times \Rddsym.
\end{aligned}
\end{equation}
Indeed, we obtain~\eqref{eq:regularized_cordes} by applying Proposition~\ref{prop:subdifferential_characterization} to $H_\lambda$, noting that the subdifferentials of~$H_\lambda$ are singleton sets containing only the values of its partial derivative.
Since the set $\Cclass$ is bounded in $\Rdd$, it follows that $H_\lambda$ is Lipschitz continuous in the second variable,
\begin{equation}\label{eq:H_reg_Lipschitz} 
|H_\lambda(x,M) - H_\lambda(x,N)| \leq C|M-N| \quad \forall (x, M, N) \in \overline{\Omega} \times \Rddsym \times \Rddsym,
\end{equation}
for some constant $C$ independent of $\lambda$.
Note that the joint continuity of $\frac{\partial H_\lambda}{\partial M}$ in both variables is in fact not an additional assumption, since it is a consequence of~\eqref{eq:H_reg_Lipschitz}, which implies the uniform boundedness of the partial derivatives, and of the continuity in both variables of $H_\lambda$, as a special case of the upper semicontinuity, in the sense of set-valued maps, of the subdifferential.

\begin{remark}
The assumption~\eqref{eq:regularized_representation} means that the approximate Hamiltonian $H_\lambda$ admits a similar representation as the Hamiltonian $H$, see~\eqref{eqn:H_def}.
This is a natural assumption since it means that the system~\eqref{eqn: regularized MFG} is connected to a mean field game with an underlying control problem. 
At a technical level, we make this assumption here since the representation of the Hamiltonian is used in the analysis of the HJB equation in Lemma~\ref{lemma: HJB existence} above.
\end{remark}

\begin{example}[Moreau--Yosida regularization]\label{example: moreau yosida}
A suitable example of a family of regular Hamiltonians satisfying Assumption~\ref{assumption: regularised hamiltonian} is provided by Moreau--Yosida regularization, also known as the Moreau envelope, where, for each $\lambda\in (0,1]$, the function $H_\lambda$ is defined by
\begin{equation}\label{eqn: moreau yosida defn}
H_\lambda(x,M) \coloneqq \min_{N \in \Rddsym} \left\{ H(x,N) + \frac{1}{2\lambda}|N-M|^2 \right\} \quad \forall (x,M)\in \overline{\Omega}\times\Rddsym. 
\end{equation}
In this case, we show that Assumption~\ref{assumption: regularised hamiltonian} holds.
First, the convexity and continuous differentiability in $M$ of $H_\lambda$ are well known, see~\cite[Theorem~6.5.7]{AubinFrankowska2009} where it is additionally shown that $\frac{\partial H_\lambda}{\partial M}(x,M)\in \partial_MH(x,M_\lambda)$ for all $(x,M)$, where $M_\lambda$ denotes the unique minimizer of~\eqref{eqn: moreau yosida defn}.
Since the subdifferential sets $\partial_M H$ are uniformly bounded, see the paragraph below~\eqref{eq:subdifferential_def}, we have $|\frac{\partial H_\lambda}{\partial M}(x,M)|\leq L_H$, which shows that the Lipschitz continuity~\eqref{eq:H_reg_Lipschitz} condition holds with a constant independent of $\lambda$.  
It is also known that~\eqref{eq:H_reg_conv} holds with~$\omega(\lambda) = \frac{L_H^2 \lambda}{2}$,  which is proved similarly to~\cite[Lemma 3.4]{OsborneSmears2025iii}, where we recall that~$L_H$ is defined in~\eqref{eq:L_H_def}.
Finally, we show that the representation formula~\eqref{eq:regularized_representation} holds by giving an explicit construction of the space $\mathcal{A}_\lambda$ and coefficients $a_\lambda$ and $f_\lambda$ in terms of the original data $\mathcal{A}$, $a$ and $f$.
In particular, let $\dsym=\frac{1}{2}d(d+1)$ denote the dimension of $\Rddsym$, and we take $\bm{\mathcal{A}}_\lambda = [ \mathcal{A}]^{\dsym+1}\times \Sigma_{\dsym+1}$, with the simplex~$\Sigma_{\dsym+1}\coloneqq \{ \theta\in \R^{\dsym+1}_+ : \sum_{i=1}^{\dsym+1} \theta_i=1\}$. 
In other words, $\bm{\mathcal{A}}_{\lambda}$ is the set of all tuples $(\alpha_1,\ldots,\alpha_{\dsym+1},\theta)$ with $\alpha_i\in \mathcal{A}$ for all $i$ and $\theta \in \Sigma_{\dsym+1}$.
The set $\bm{\mathcal{A}}_\lambda$ is then a compact metric space when equipped with a suitable metric, see Appendix~\ref{sec:Morea-Yosida_representation} for further details.
Then, we define the functions
\[
a_\lambda(x,\bm{\alpha})\coloneqq \sum_{i=1}^{\dsym+1}\theta_ia(x,\alpha_i), \quad f_{\lambda}(x,\bm{\alpha})\coloneqq \sum_{i=1}^{\dsym+1}\theta_i f(x,\alpha_i) + \frac{\lambda}{2}|a_\lambda(x,\bm{\alpha})|^2 \quad \forall (x,\bm{\alpha})\in \overline{\Omega}\times \bm{\mathcal{A}}_\lambda.
\]
Note that the space $\bm{\mathcal{A}}_\lambda $ and the function $a_{\lambda}$ are in fact independent of $\lambda$ in this case.
It is clear that the functions $a_\lambda$ and $f_\lambda$ are continuous over $\overline{\Omega}\times\bm{\mathcal{A}}_\lambda$.
Moreover, since $a_\lambda(x,\bm{\alpha})$ is a convex combination of values of the function $a$ and since $\Cclass$ is convex, it is clear that $a_\lambda(x,\bm{\alpha})\in \Cclass$ for all $(x,\bm{\alpha})\in \overline{\Omega}\times \bm{\mathcal{A}}_\lambda$.
Then, for all $\lambda>0$, the function $H_\lambda$ defined by~\eqref{eqn: moreau yosida defn} satisfies~\eqref{eq:regularized_representation}, i.e.\
\begin{equation}\label{eq:moreau_yosida_representation}
H_\lambda(x,M) = \sup_{\bm{\alpha}\in \bm{\mathcal{A}}_\lambda} \left\{ -a_\lambda(x,\bm{\alpha}):M-f_\lambda(x,\bm{\alpha}) \right\} \quad \forall (x,M)\in\overline{\Omega}\times\Rddsym.
\end{equation}
We give  an elementary proof of~\eqref{eq:moreau_yosida_representation} in Appendix~\ref{sec:Morea-Yosida_representation} below.
This shows that Assumption~\ref{assumption: regularised hamiltonian} is satisfied in the case of Moreau--Yosida regularization.
\end{example}

Theorem~\ref{thm: MFG existence}, with $H$ replaced by $H_\lambda$, shows that there exists a solution pair $(u_\lambda, m_\lambda) \in V \times L_+^2(\Omega)$ of~\eqref{eqn: regularized MFG} for each $\lambda\in (0,1]$, where differentiability of $H_\lambda$ w.r.t.\ $M$ implies that the KFP equation~\eqref{eqn: regularized MFG2} is a differential equation rather than merely a differential inclusion.

Our main result regarding regularization is the following theorem.

\begin{theorem}\label{thm: regularization convergence}
Let~$\Omega \subset \R^d$ be an open, bounded, convex domain.
Let $\mathcal{A}$ be a compact metric space, let $a \in C(\overline{\Omega} \times \mathcal{A};\Rddsymp)$ satisfy Assumption~\ref{assumption: Cordes}, and let $f \in C(\overline{\Omega}\times\mathcal{A})$.
Let $F \colon L^2(\Omega) \rightarrow L^2(\Omega)$ be completely continuous.
Let $G \in V^*$ be nonnegative.
Let~$\{H_\lambda\}_{\lambda \in (0,1]}$ be a family of regularizations of \eqref{eqn:H_def} satisfying Assumption~\ref{assumption: regularised hamiltonian}.
Then, there exists a subsequence $\{(u_{\lambda_j}, m_{\lambda_j})\}_{j \in \mbb{N}}$ of solutions of~\eqref{eqn: regularized MFG} for $\lambda=\lambda_j$, with $\lambda_j\tends 0$ as $j\tends \infty$, and a solution $(u,m)\in V\times L^2_+(\Omega)$ of~\eqref{eq:MFG_PDI}, equivalently of~\eqref{eq:MFG_VI}, such that $u_{\lambda_j} \rightarrow u$ strongly in~$V$, and $m_{\lambda_j} \rightharpoonup m$ weakly in~$L^2(\Omega)$, as $j\to \infty$.
\end{theorem}
\begin{proof}
\begin{subequations}
By Theorem~\ref{thm: MFG existence}, for each~$\lambda\in(0,1]$, there exists a pair $(u_\lambda, m_\lambda) \in V \times L^2_+(\Omega)$  that solves 
\begin{alignat}{2}
H_{\lambda}[u_{\lambda}] &= F[m_{\lambda}] &\quad&\text{a.e. in } \Omega, \label{eqn: regularization pf1}
\\\angled{G, v-u_{\lambda}}_{V^* \times V} &\leq \left( m_{\lambda},  H_{\lambda}[v] - H_{\lambda}[u_{\lambda}]\right)_{\Omega} &\quad& \forall v \in V. \label{eqn: regularization pf2}
\end{alignat}
\end{subequations}
Moreover, the uniform bound in~\eqref{eq:solution_boundedness_m} implies that $\{m_\lambda\}_{\lambda \in (0,1]}$ is uniformly bounded in $L^2(\Omega)$.
Hence, there exists a~${m} \in L^2(\Omega)$ and a sequence~$\{\lambda_j\}_{j \in \mbb{N}}$, with~$\lambda_j \rightarrow 0$ as $j \rightarrow \infty$, such that~$m_{\lambda_j} \rightharpoonup {m}$ in $L^2(\Omega)$.
Mazur's theorem implies that~${m} \geq 0$ a.e.\ in~$\Omega$.
Let $u \in V$ denote the unique solution of $H[u] = F[m]$ a.e.\ in $\Omega$, which exists by Lemma~\ref{lemma: HJB existence}.
To show that the $u_{\lambda_j}$ converge strongly to $u$ in~$V$, we may rewrite~\eqref{eqn: regularization pf1} for $\lambda=\lambda_j$ as
\[ H[u_{\lambda_j}] = F[m_{\lambda_j}] + H[u_{\lambda_j}] - H_{\lambda_j}[u_{\lambda_j}], \]
and hence we may use Lemma~\ref{lemma: HJB existence} to find that
\begin{equation}
\|{u} - u_{\lambda_j}\|_{H^2(\Omega)} \leq C \left( \|F[{m}] - F[m_{\lambda_j}]\|_{L^2(\Omega)} + \omega({\lambda_j}) \right),
\end{equation}
where we have also used~\eqref{eq:H_reg_conv}.
Thus, by complete continuity of~$F$, we find that $u_{\lambda_j} \rightarrow {u}$ strongly in $V$ as~$j \rightarrow \infty$.
We now show how to pass to the limit in~\eqref{eqn: regularization pf2} to deduce that the pair~$({u},{m}) \in V \times L_+^2(\Omega)$ solves~\eqref{eq:MFG_VI}.
Indeed, for any test function $v \in V$, we may write
\begin{multline*}
\left( m_{\lambda_j}, H_{\lambda_j}[v] - H_{\lambda_j}[u_{\lambda_j}] \right)_\Omega = \left( m_{\lambda_j}, H[v] - H[{u}] \right)_\Omega 
\\ + \left( m_{\lambda_j}, H_{\lambda_j}[v] - H[v] \right)_\Omega + \left(m_{\lambda_j}, H[{u}] - H_{\lambda_j}[u_{\lambda_j}] \right)_\Omega,
\end{multline*}
and we note that the last two terms on the right-hand side in the equation above vanish in the limit as $j \rightarrow \infty$, since the $\|m_{\lambda_j}\|_{L^2(\Omega)}$ are uniformly bounded, and since $\|H_{\lambda_j}[v] - H[v] \|_{L^2(\Omega)} $ and $\|H[{u}] - H_{\lambda_j}[u_{\lambda_j}]\|_{L^2(\Omega)}$ vanish in the limit by~\eqref{eq:H_reg_conv} and by $u_{\lambda_j} \to u$ in $V$.
Therefore, one may pass to the limit in~\eqref{eqn: regularization pf2} to find that $ \angled{G, v - {u}}_{V^* \times V} \leq \left({m}, H[v] - H[{u}] \right)_\Omega$ for all $v \in V$. This shows that the pair~$(u, m) \in V \times L^2_+(\Omega)$ is a solution of the VI problem~\eqref{eq:MFG_VI}, and, equivalently, of the PDI problem~\eqref{eq:MFG_PDI} by Theorem~\ref{thm:PDI_VI_equivalence}.
\end{proof}

\begin{remark}
\label{remark: regularization whole sequence}
In Theorem~\ref{thm: regularization convergence}, there is no assumption of strict monotonicity on $F$, hence solutions of both the original and regularized problems may be nonunique in general.
This is why convergence is only shown up to subsequences.
However, if one adds the assumption of strict monotonicity of~$F$, then, by Theorem~\ref{thm:MFG_uniqueness}, uniqueness of solutions holds for both the original and regularized problems, and convergence of the whole sequence is recovered.
\end{remark}

\begin{remark}
It is interesting to contemplate if one can also prove Theorem~\ref{thm: regularization convergence} without using the VI formulation of the problem.
The difficulty is that there are examples showing that, generally, the sequence of functions $\frac{\partial H_{\lambda_j}}{\partial M}(\cdot,D^2u_{\lambda_j})$ need not contain any strongly converging subsequence, in any $L^p$ norm with $p\geq 1$, to some element of $\mathcal{D}H[u]$, even when $u_{\lambda_j}\to u$ in $V$.
This difficulty is similar to the one mentioned in Remark~\ref{rem:no_strong_continuity} above.
It is then hard to pass to limits directly in the KFP equation~\eqref{eqn: regularized MFG2}.
This further highlights the benefit of working with the VI formulation of the problem when passing to limits.
\end{remark}

\subsection{Perturbations of~$G$}

In the previous subsection, we considered the effect of regularization of the Hamiltonian. One can also consider perturbations of the other problem data.
To illustrate this, we end this section with a short proof of continuous dependence of solutions, up to subsequences, on the datum~$G$.

\begin{proposition}\label{proposition: cts dependence G}
Let~$\Omega \subset \R^d$ be an open, bounded, convex domain.
Let $\mathcal{A}$ be a compact metric space, let $a \in C(\overline{\Omega} \times \mathcal{A};\Rddsymp)$ satisfy Assumption~\ref{assumption: Cordes}, and let $f \in C(\overline{\Omega}\times\mathcal{A})$.
Let $F \colon L^2(\Omega) \rightarrow L^2(\Omega)$ be completely continuous.
Let $G \in V^*$ be nonnegative.
Let $\{G_n\}_{n \in \N}$ be a sequence of nonnegative functionals in $V^*$ that is weak-$*$ converging to $G$ in $V^*$.
For each $n\in \N$, let $(u_n, m_n) \in V \times L^2_+(\Omega)$ be a solution of the VI problem
\begin{subequations}\label{eq:G_perturb_MFG}
\begin{alignat}{2}
H[u_n]&= F[m_n] &\quad& \text{in } \Omega, \label{eq:G_perturb_HJB}
\\ \pair{G_n}{v-u_n}_{V^*\times V} &\leq \left(m_n, H[v]-H[u_n]\right)_\Omega &\quad& \forall v \in V. \label{eq:G_perturb_KFP} 
\end{alignat}
\end{subequations}
Then there exists a solution $(u,m)\in V \times L^2_+(\Omega)$ of the VI problem~\eqref{eq:MFG_VI}, and a subsequence $\{(u_{n_j},m_{n_j})\}_{j\in\N}$ such that $u_{n_j} \rightarrow u$ strongly in $V$, and $m_{n_j} \rightharpoonup m$ weakly in $L^2(\Omega)$, as $j \rightarrow \infty$.
\end{proposition}
\begin{proof}
Firstly, since we have that~$G_n \overset{*}{\rightharpoonup} G$ in~$V^*$, one finds that~$\{G_n\}_{n \in \mbb{N}}$ is uniformly bounded in $V^*$.
Hence by using~\eqref{eqn: KFP L2 control} we find that the sequence~$\{m_n\}_{n \in \mbb{N}}$ is uniformly bounded in $L^2(\Omega)$, and thus there exists a weakly converging subsequence~$\{m_{n_j}\}_{j \in \mbb{N}}$ such that~$m_{n_j} \rightharpoonup {m}$ in~$L^2(\Omega)$ as $j \rightarrow \infty$.
By Mazur's theorem we again see that $m\geq 0$ almost everywhere in $\Omega$.
If we now define~$u \in V$ to be the unique solution to $H[u] = F[m]$, then, from~\eqref{eqn: HJB stability} and the complete continuity of~$F$, one finds that $u_{n_j} \rightarrow u$ strongly in $V$ as $j \rightarrow \infty$.
By using the weak-$*$ convergence~$G_{n_j} \overset{*}{\rightharpoonup} G$ in $V^*$, the strong convergence~$u_{n_j} \rightarrow u$ in $V$, and the weak convergence $m_{n_j} \rightharpoonup m$ in $L^2(\Omega)$, we may pass to the limit $j \rightarrow \infty$ in~\eqref{eq:G_perturb_KFP}.
Recalling the definition of~$u$, we then conclude that $(u, m) \in V \times L^2_+(\Omega)$ is a solution to the VI problem~\eqref{eq:MFG_VI}.
\end{proof}

Similar to before, if the operator $F$ is additionally assumed to be strictly monotone, then the conclusion of Proposition~\ref{proposition: cts dependence G} above can be strengthened to convergence of the whole sequence.

\section*{Conclusion}
We have shown the existence of solutions of fully nonlinear second-order MFG, on general bounded convex domains, with uniformly elliptic coefficients satisfying the Cordes condition.
Solutions are additionally unique if the coupling is strictly monotone.
The key challenges addressed include the possible nondifferentiability of the Hamiltonian, the lack of smoothness of the domain boundary, the dependence of the Hamiltonian on the whole Hessian of the value function, and also the absence of any potential structure for the system.
We have also obtained results on the continuous dependence of the solution, where we have shown that the MFG system for the problem with nondifferentiable Hamiltonian arises as the limit of problems with differentiable Hamiltonians. 
In so doing, we have obtained several results of independent interest that play a role in the analysis, such as the comparison principle for nondivergence-form PDE with discontinuous coefficients, and their adjoint KFP equations.

A key original ingredient of our approach is the nonstandard variational inequality formulation of the problem, which was shown to be equivalent to the partial differential inclusion formulation.
We expect that the approach based on the equivalence of PDI and VI formulations will continue to provide useful tools in the analysis of a wide range of problems in mean field games, since it enables passages to the limits in a broad range of circumstances.
Possible further directions for investigation include problems with lower order terms, time-dependent problems, the case of nonseparable Hamiltonians, first-order problems, as well as the development of numerical methods.
MFG systems with other types of boundary conditions, such as ergodic problems with no-flux conditions on the boundary, also present interesting additional challenges for further work, owing to the possible low regularity of diffusion coefficients that enter into the boundary conditions in the setting of fully nonlinear nondifferentiable Hamiltonians.

\section*{Acknowledgements}
Thomas~Sales was supported by the Engineering and Physical Sciences Research Council whilst at the University of Sussex [grant number EP/Z535138/1], and at University College London [grant number UKRI4977].
Iain~Smears was supported by the Engineering and Physical Sciences Research Council [grant number EP/Y008758/1].
We thank T.~Sprekeler and E.~S\"uli for bringing to our attention the reference~\cite{Kadlec1964}, along with its translation.

\appendix
\section{Comparison principle for nondivergence form elliptic operators with Cordes coefficients}
\label{appendix: comparison principle}

We now prove the comparison principle for nondivergence form elliptic equations with Cordes coefficients, as detailed in~Theorem~\ref{thm:comparison_nondivergence}.
\begin{proof}[Proof of Theorem~\ref{thm:comparison_nondivergence}]
The result is trivial in the case $d=1$, so it remains to consider the case $d\geq 2$.
We prove the result by a regularization argument.
We first extend~$\overline{a}$ from~$\Omega$ to the whole of~$\R^d$ by setting $\overline{a}|_{\R^d\setminus \Omega} \coloneqq \overline{\nu}I_d$ where $I_d$ denotes the $d\times d$ identity matrix. 
It is clear that $\overline{\nu}I_d\in \Cclass$, so $\overline{a}(x)\in \Cclass$ for almost every $x\in \R^d$.
For each $\rho>0$, let $\phi_\rho \in C^\infty_0(\R^d)$ denote the standard mollifier supported in the closed ball of radius~$\rho$. 
Since the standard mollifier~$\phi_\rho$ is nonnegative with $\int_{\R^d}\phi_\rho(y)\,\mathrm{d}y=1$, and since~$\Cclass$ is closed and convex, it follows that the mollification $\overline{a}_{\rho}$, which is defined by $\overline{a}_{\rho}(x) \coloneqq \int_{\R^d} \overline{a}(y) \phi_\rho(x-y) \,\mathrm{d}y$ for each $x\in\R^d$, also satisfies
\begin{equation}\label{eq:comparison_nondivergence_5}
\overline{a}_{\rho}(x) \in \Cclass \quad \forall x\in \R^d,
\end{equation}
since mollification preserves convexity of sets.
We will also regularize the domain~$\Omega$ by using the following result from \cite[Theorem~A.1]{SmearsThesis}: for each $\delta >0$, there exists an open convex set $\Omega_{\delta}$ with $C^{1,1}$ boundary such that $\Omega_\delta \subset \Omega$ and $\mathrm{dist}(\Omega,\Omega_\delta)<\delta$, where
\begin{equation}
\mathrm{dist}(\Omega,\Omega_\delta) \coloneqq \sup_{x\in\Omega}\inf_{y\in\Omega_{\delta}}|x-y|+\sup_{y\in\Omega_{\delta}}\inf_{x\in\Omega}|x-y|.
\end{equation}
Also, for each $\delta>0$, we can find a~$\rho_\delta>0$ so that $\norm{\overline{a}-\overline{a}_{\rho_\delta}}_{L^2(\Omega;\R^{d\times d})}<\delta$.
We also choose $g_\delta \in L^d(\Omega)$ to be a nonnegative approximation of~$g$ that satisfies $\norm{g-g_\delta}_{L^2(\Omega)}<\delta$.
Consider now, for each $\delta>0$, the regularized problem
\begin{equation}\label{eq:comparison_nondivergence_1}
\begin{aligned}
 - \overline{a}_{\rho_\delta}:D^2v_\delta &= g_\delta && \text{in } \Omega_\delta, 
 \\ v_\delta &=0 &&\text{on }\partial\Omega_\delta.
\end{aligned}
\end{equation}
This is a case of a nondivergence form elliptic equation on a bounded domain with $C^{1,1}$ boundary with $\overline{a}_{\rho_\delta}\in C(\overline{\Omega_\delta})$ and $g_\delta \in L^d(\Omega)$.
Therefore, by the Calder\'on--Zygmund theory of nondivergence elliptic equations, see \cite[Theorem~9.15]{GilbargTrudinger2001}, we deduce that, for each $\delta>0$, there exists a unique $v_\delta\in W^{2,d}(\Omega_\delta)\cap W^{1,d}_0(\Omega_\delta)$ that solves~\eqref{eq:comparison_nondivergence_1}.
Recall that $W^{2,d}(\Omega_\delta)$ is continuously embedded in~$C(\overline{\Omega_\delta})$.
Since $g_\delta$ is nonnegative a.e.\ in $\Omega$, and hence also in~$\Omega_\delta$, it follows from the Aleksandrov--Bakelman--Pucci maximum principle~\cite[Theorem~9.1]{GilbargTrudinger2001} that $v_{\delta}$ is nonnegative a.e.\ in $\Omega_\delta$.

Since $\overline{a}_{\rho_\delta}(x) \in \Cclass$ for all $x\in\Omega_\delta$ and since the functions $g_\delta$ form a uniformly bounded subset of $L^2(\Omega)$, and since $\Omega_\delta$ is convex, \cite[Theorem~3]{SmearsSuli2013} shows that there exists a constant $C$, independent of $\delta$, such that
\begin{equation}\label{eq:comparison_nondivergence_2}
\norm{v_\delta}_{H^2(\Omega_\delta)} \leq C \quad\forall\delta>0.
\end{equation}
Observe that in the statement of \cite[Theorem~3]{SmearsSuli2013}, the constant of the a priori bound is allowed to depend on the diameter of the domain, namely~$\diam(\Omega_\delta)$ in the current context, yet by tracing its explicit form in the proof, it is found that it can be made independent of $\delta$ since $\diam(\Omega_\delta)\leq \diam(\Omega)$ for all $\delta>0$.
Let $\overline{v}_{\delta}$ denote the zero-extension of $v_\delta$ to the whole of $\Omega$. It follows from $v_\delta \in H^1_0(\Omega_\delta)$ that $\overline{v}_\delta \in H^1_0(\Omega)$ for each $\delta>0$ (see for instance~\cite[Chapter 5]{AdamsFournier03}), and also from~\eqref{eq:comparison_nondivergence_2} that the $\overline{v}_\delta$ are uniformly bounded in $H^1_0(\Omega)$.
Furthermore, the $\overline{v}_\delta$ are nonnegative almost everywhere in $\Omega$ for all $\delta>0$.
Let $M_{\delta} \in L^2(\Omega;\Rddsym)$ denote the zero-extension of $D^2 v_{\delta}$ to $\Omega$. It follows that the $M_\delta$ are also bounded in $L^2(\Omega;\Rddsym)$ uniformly w.r.t.\ $\delta$.
Then, after possibly passing to a subsequence, we may assume, without loss of generality, that there exists a $\overline v \in H^1_0(\Omega)$ and an $M\in L^2(\Omega;\Rddsym)$ such that $\overline{v}_\delta \rightharpoonup \overline{v}$ in $H^1_0(\Omega)$ and $M_\delta \rightharpoonup M$ in $L^2(\Omega;\Rddsym)$ as $\delta\to 0$. 
We claim that $\overline{v}=v_g$ is the unique strong solution in $V$ of
\begin{equation}\label{eq:comparison_nondivergence_0}
\begin{aligned}
 - \overline{a}:D^2v_g &= g && \text{in } \Omega, 
 \\ v_g &=0 &&\text{on }\partial\Omega.
\end{aligned}
\end{equation}
To verify this, let $\psi \in C^\infty_0(\Omega)$ be arbitrary, and let $K\coloneqq \supp \psi$ denote its support. Since $K$ is a compact subset of $\Omega$ and since $\mathrm{dist}(\Omega,\Omega_\delta)<\delta$, with $\Omega_\delta$ and $\Omega$ convex, it follows that, for all~$\delta$ sufficiently small, we have $K\subset \Omega_\delta \subset \Omega$, see~\cite[Lemma A.4]{SmearsThesis}. This implies that, for any $i,j\in \{1,\ldots,d\}$, we have
\begin{multline}\label{eq:comparison_nondivergence_3}
\int_\Omega M_{ij}  \psi \,\dx 
= \lim_{\delta\to 0} \int_\Omega \left(M_\delta\right)_{ij}  \psi \, \dx  = \lim_{\delta\to 0} \int_{\Omega_{\delta}} (\partial_{x_{i} x_{j}} v_{\delta}) \psi \,\dx
\\ = - \lim_{\delta \rightarrow 0} \int_{\Omega_{\delta}} \partial_{x_i} v_\delta \partial_{x_j} \psi \, \dx = - \lim_{\delta\to 0} \int_\Omega \partial_{x_i} \overline{v}_\delta \,\partial_{x_j}\,\psi \,\dx = - \int_\Omega \partial_{x_i} \overline{v} \,\partial_{x_j} \psi \,\dx,
\end{multline}
where, in the identities above, we have first used the weak convergence $M_\delta\rightharpoonup M$, followed by the fact that $M_\delta=D^2 v_{\delta}$ on $K=\supp \psi$ for all $\delta$ sufficiently small, and finally we have used the weak convergence $\overline{v}_{\delta}\rightharpoonup \overline{v}$ in $H^1_0(\Omega)$ as $\delta\to 0$.
Since $\psi$ was arbitrary, the equation~\eqref{eq:comparison_nondivergence_3} implies that $\overline{v}\in H^2(\Omega)$ and that $D^2 \overline{v} = M$ in $\Omega$.
This shows that $\overline{v}\in V=H^2(\Omega)\cap H^1_0(\Omega)$.
Then, it follows from~\eqref{eq:comparison_nondivergence_1}, from the strong convergence $\overline{a}_{\rho_\delta}\to \overline{a}$ in $L^2(\Omega)$, and $g_\delta\to g$ in $L^2(\Omega)$ as $\delta\to 0$, that
\[
\left( - \overline{a}:D^2 \overline{v},\psi\right)_\Omega = \lim_{\delta\to 0} (-\overline{a}_{\rho_\delta}:M_\delta, \psi)_\Omega = 
\lim_{\delta\to 0}( -\overline{a}_{\rho_\delta}:D^2v_{\delta}, \psi)_\Omega = \lim_{\delta\to 0} (g_\delta,\psi)_\Omega = (g,\psi)_\Omega,
\]
where we use again the fact that $K\subset \Omega_\delta$ for all $\delta$ sufficiently small.
Notice that, in the first equality above, we use the weak convergence $M_\delta\rightharpoonup M=D^2\overline{v}$ in $L^2(\Omega; \Rddsym)$ along with the strong convergence $\overline{a}_{\rho_\delta}\to \overline{a}$ in $L^2(\Omega;\R^{d\times d})$ as $\delta\to 0$, and in the last equality we use the convergence $g_\delta\to g$ in $L^2(\Omega)$ as $\delta \to 0$.
Since $\psi$ above was arbitrary, we conclude that $\overline{v}\in H^2(\Omega)\cap H^1_0(\Omega)$ solves $-\overline{a}:D^2 \overline{v} = g$ pointwise a.e.\ in $\Omega$.
Therefore $\overline{v}=v_g$ is the unique solution of~\eqref{eq:comparison_nondivergence_0}.
To conclude the proof, recall that $\overline{v}_\delta\geq 0$ a.e.\ in $\Omega$ for all $\delta>0$, so the convergence $\overline{v}_\delta\rightharpoonup \overline{v}=v_g$ as $\delta\to 0$ implies that $v_g\geq 0$ a.e.\ in $\Omega$ by Mazur's theorem.
\end{proof}

\section{Representation of Moreau--Yosida regularizations}\label{sec:Morea-Yosida_representation}

We show here the representation formula for Moreau--Yosida regularization that was stated in Example~\ref{example: moreau yosida}.
In fact, this result is likely to be of independent interest in other contexts, since the formula holds for more general classes of Hamiltonians than the one considered in~\eqref{eqn:H_def}, as the matrix structure of the arguments in~$H$ plays no role in the results of this appendix.
We therefore state and prove it in a sufficiently general setting.
Let $N\in \N$ be a positive integer, and consider the function $h\colon \overline{\Omega}\times\R^N\tends \R$ defined by
\begin{equation}\label{eq:little_h_def}
h(x,z)\coloneqq \sup_{\alpha\in \mathcal{A}}\left\{ b(x,\alpha)\cdot z - g(x,\alpha) \right\} \quad \forall (x,z)\in \overline{\Omega}\times \R^N,
\end{equation}
where $\mathcal{A}$ is some given compact metric space, and $b\colon \overline{\Omega}\times \mathcal{A}\tends \R^N$ and $g\colon\overline{\Omega}\times \mathcal{A} \tends \R$ are given continuous functions.
The function~$h$ is then real-valued, continuous, and convex in its second argument.
Moreover, the supremum in~\eqref{eq:little_h_def} is achieved.
Our aim is to show that the Moreau--Yosida regularization $h_\lambda$, $\lambda>0$, defined by
\begin{equation}\label{eq:little_h_moreau}
h_\lambda(x,z) \coloneqq \min_{z^\prime\in \R^N }\left\{ h(x,z^\prime)+\frac{1}{2\lambda}|z^\prime-z|^2 \right\} \quad \forall (x,z)\in\overline{\Omega} \times \R^N,
\end{equation}
admits a similar representation as $h$, i.e.\ as a supremum of a set of affine functions in $z$ over some compact metric space, where the coefficients entering the affine functions are continuous jointly in both the spatial and control variables.

We start by recalling some known results.
First, following the same arguments as in the proof of Proposition~\ref{prop:subdifferential_characterization}, the partial subdifferential $\partial_z h$ is characterized by
\begin{equation}\label{eq:little_h_subdifferential}
\partial_z h(x,z) = \Conv \left\{ b(x,\alpha_*) : \alpha_*\in \argmax_{\alpha\in\mathcal{A}}\{ b(x,\alpha)\cdot z-g(x,\alpha) \} \right\} \quad \forall (x,z)\in \overline{\Omega}\times \R^N,
\end{equation}
where we note that the set $\argmax_{\alpha\in\mathcal{A}}\{ b(x,\alpha)\cdot z-g(x,\alpha) \}$ is nonempty for any $(x,z)\in \overline{\Omega}\times \R^N$.
We also recall Carath\'eodory's Theorem, see~\cite[Theorem~2.29, p.~55]{RockafellarWets1998}, which shows that, for any nonempty set $C\subset \R^{N}$, each point of the convex hull $\Conv C$ can be represented as a convex combination of at most $N+1$ points of $C$. 

We now detail the explicit construction of the supremum representation of the Moreau--Yosida regularizations.
Let $\Sigma_{N+1}$ denote the standard probability simplex on $\R^{N+1}$, i.e.\ $\Sigma_{N+1}\coloneqq \{ \theta \in \R^{N+1}_+ : \sum_{i=1}^{N+1}\theta_i=1\}$, where $\R^{N+1}_+$ is the set of vectors in $\R^{N+1}$ that are nonnegative componentwise.
Let $\bm{\mathcal{A}}_*$ denote the set of all tuples $(\alpha_1,\ldots,\alpha_{N+1},\theta )$ with $\alpha_i\in \mathcal{A}$ for $i\in\{1,\ldots,N+1\}$ and $\theta\in \Sigma_{N+1}$.
In more succinct terms, 
\begin{equation}\label{eq:relaxed_controls_def}
\bm{\mathcal{A}}_*\coloneqq [\mathcal{A}]^{N+1}\times \Sigma_{N+1} .
\end{equation}
The set~$\bm{\mathcal{A}}_*$ is a compact metric space when equipped with the metric defined by $
\rho_{\bm{\mathcal{A}}_*}(\bm{\alpha},\bm{\beta}) \coloneqq \sum_{i=1}^{N+1} \rho_{\mathcal{A}}(\alpha_i,\beta_i)+|\theta-\eta|$ for all $\bm{\alpha}=(\alpha_1,\ldots,\alpha_{N+1},\theta)$ and $\bm{\beta}=(\beta_1,\ldots,\beta_{N+1},\eta)$ in $\bm{\mathcal{A}}_*$, where~$\rho_{\mathcal{A}}$ denotes the metric on~$\mathcal{A}$.
Let the functions  $b_{*}\colon \overline{\Omega}\times \bm{\mathcal{A}}_*\tends \R^N$ and $g_\lambda \colon \overline{\Omega}\times \bm{\mathcal{A}}_*\tends \R $, $\lambda\geq 0$, be defined by
\begin{equation}\label{eq:yosida_coeff_def}
\begin{aligned}
b_*(x,\bm{\alpha})\coloneqq \sum_{i=1}^{N+1} \theta_i b(x,\alpha_i),
&&& g_\lambda(x,\bm{\alpha}) \coloneqq \sum_{i=1}^{N+1} \theta_i g(x,\alpha_i)+\frac{\lambda}{2}|b_*(x,\bm{\alpha})|^2.
\end{aligned}
\end{equation}
Note that~$b_*$ is independent of~$\lambda$. 
In fact the regularization parameter~$\lambda$ enters only in the function~$g_\lambda$ defined above.
The functions~$b_*$ and~$g_\lambda$ are continuous over $\overline{\Omega}\times \bm{\mathcal{A}}_*$ for all $\lambda\geq 0$ since the functions $b$ and $g$ are assumed to be continuous over~$\overline{\Omega}\times \mathcal{A}$.

We now show that the Moreau--Yosida regularization admits the desired representation.

\begin{proposition}[Representation of Moreau--Yosida Regularization]\label{prop:yosida_representation}
Let $\Omega\subset \R^d$ be a bounded open set, let $\mathcal{A}$ be a compact metric space, and let the function $h$ be defined by~\eqref{eq:little_h_def} with $(b,g)\colon \overline{\Omega}\times\mathcal{A}\tends \R^N\times \R$ continuous.
For each $\lambda>0$, let $h_\lambda$ be defined by~\eqref{eq:little_h_moreau}. 
Let the compact metric space $\bm{\mathcal{A}}_*$ be defined by~\eqref{eq:relaxed_controls_def} and let the functions~$b_*$ and~$g_\lambda$ be defined by~\eqref{eq:yosida_coeff_def}.
Then
\begin{equation}\label{eq:yosida_representation}
h_\lambda(x,z) = \sup_{\bm{\alpha}\in \bm{\mathcal{A}}_*} \left\{ b_*(x,\bm{\alpha})\cdot z -g_\lambda(x,\bm{\alpha}) \right\} \quad \forall (x,z)\in \overline{\Omega}\times\R^N, \quad \forall \lambda>0.
\end{equation}
\end{proposition}
Note that the supremum on the right-hand side of~\eqref{eq:yosida_representation} is achieved since $\bm{\mathcal{A}}_*$ is compact and the functions~$b_*$ and~$g_\lambda$ are continuous.
\begin{proof}
Let the function $\widetilde{h}_\lambda$ be defined as the expression on the right-hand side of~\eqref{eq:yosida_representation}, i.e.\ let 
\begin{equation}\label{eq:yosida_representation_0}
\widetilde{h}_{\lambda}(x,z)\coloneqq \sup_{\bm{\alpha}\in \bm{\mathcal{A}}_*} \left\{ b_*(x,\bm{\alpha})\cdot z -g_\lambda(x,\bm{\alpha}) \right\}
\quad \forall(x,z)\in\overline{\Omega}\times \R^N.
\end{equation}
We will prove~\eqref{eq:yosida_representation} in two steps, first by showing that $\widetilde{h}_{\lambda} \leq h_\lambda$, and afterwards that $h_\lambda \leq \widetilde{h}_\lambda$.
To show the first inequality, consider an arbitrary $\bm{\alpha}\in \bm{\mathcal{A}}_*$, a fixed but arbitrary $(x,z)\in \overline{\Omega}\times \R^N$, and $\lambda>0$.
Then, for any $z^\prime \in \R^N$, we use Young's inequality to find that
\begin{equation}
\begin{split}
b_*(x,\bm{\alpha})\cdot z - g_\lambda(x,\bm{\alpha})
& \underset{\phantom{\eqref{eq:yosida_representation_3}}}{=} b_*(x,\bm{\alpha})\cdot z^\prime + b_*(x,\bm{\alpha})\cdot(z-z^\prime)-g_\lambda(x,\bm{\alpha}) 
\\ & \underset{\phantom{\eqref{eq:yosida_representation_3}}}{\leq} b_*(x,\bm{\alpha})\cdot z^\prime +\frac{1}{2\lambda}|z-z^\prime|^2+\frac{\lambda}{2}|b_*(x,\bm{\alpha})|^2-g_\lambda(x,\bm{\alpha})
\\ & \underset{\eqref{eq:yosida_coeff_def}}{=} \sum_{i=1}^{N+1}\theta_i \left[ b(x,\alpha_i)\cdot z^\prime-g(x,\alpha_i) \right]+\frac{1}{2\lambda}|z-z^\prime|^2
\\ & \underset{\eqref{eq:little_h_def}}{\leq} \sum_{i=1}^{N+1}\theta_i h(x,z^\prime) +\frac{1}{2\lambda}|z-z^\prime|^2 
\\ & \underset{\phantom{\eqref{eq:yosida_representation_3}}}{=} h(x,z^\prime) +\frac{1}{2\lambda}|z-z^\prime|^2.
\end{split}
\end{equation}
Since the inequality above holds for arbitrary $\bm{\alpha}\in\bm{\mathcal{A}}_*$ and $z^\prime\in \R^N$, we deduce that $\widetilde{h}_{\lambda}\leq h_\lambda$, for all $\lambda >0$.
To show the converse inequality, still considering a fixed but arbitrary $(x,z)\in \overline{\Omega}\times \R^N$ and $\lambda>0$, we use the well-known fact that the minimum in the right-hand side of~\eqref{eq:little_h_moreau} is achieved by some $z_\lambda\in \R^N$, and that
\begin{equation}\label{eq:yosida_representation_1}
\frac{z-z_\lambda}{\lambda} \in \partial_z h(x,z_\lambda),
\end{equation}
see \cite[Theorem~6.5.7]{AubinFrankowska2009} for instance.
Recalling~\eqref{eq:little_h_subdifferential}, we then apply Carath\'eodory's Theorem to deduce that there exists some $\bm{\alpha}=(\alpha_1,\ldots,\alpha_{N+1},\theta)\in\bm{\mathcal{A}}_*$, with $\alpha_i\in \argmax_{\alpha\in\mathcal{A}}\{ b(x,\alpha)\cdot z_\lambda-g(x,\alpha) \}$ for all $i\in\{1,\ldots,N+1\}$, such that
\begin{equation}\label{eq:yosida_representation_2}
\frac{z-z_\lambda}{\lambda} = \sum_{i=1}^{N+1}\theta_i b(x,\alpha_i) = b_*(x,\bm{\alpha}).
\end{equation}
Using the fact that the $\alpha_i$ are maximizers of $b(x,\alpha)\cdot z_\lambda-g(x,\alpha)$ over $\mathcal{A}$, we have $h(x,z_\lambda) = b(x,\alpha_i)\cdot z_\lambda - g(x,\alpha_i)$ for each $i\in\{1,\ldots,N+1\}$, and hence
\begin{equation}\label{eq:yosida_representation_3}
 h(x,z_\lambda) = \sum_{i=1}^{N+1}\theta_i \left[b(x,\alpha_i)\cdot z_\lambda-g(x,\alpha_i) \right]  = b_*(x,\bm{\alpha})\cdot z_\lambda-\sum_{i=1}^{N+1}\theta_i g(x,\alpha_i).
\end{equation}
Therefore, we find that
\begin{equation}
\begin{split}
h_\lambda(x,z) &\underset{\phantom{\eqref{eq:yosida_representation_3}}}{=}h(x,z_\lambda)+\frac{1}{2\lambda}|z-z_\lambda|^2
\\ &\underset{\eqref{eq:yosida_representation_2}}{=} h(x,z_\lambda) + b_*(x,\bm{\alpha})\cdot (z-z_\lambda) - \frac{\lambda}{2}|b_*(x,\bm{\alpha})|^2
\\ &\underset{\eqref{eq:yosida_representation_3}}{=} b_*(x,\bm{\alpha})\cdot z - \sum_{i=1}^{N+1} \theta_i g(x,\alpha_i) - \frac{\lambda}{2}|b_*(x,\bm{\alpha})|^2
\\ &\underset{\phantom{\eqref{eq:yosida_representation_3}}}{=} b_*(x,\bm{\alpha})\cdot z - g_\lambda(x,\bm{\alpha})
\underset{\eqref{eq:yosida_representation_0}}{\leq} \widetilde{h}_{\lambda}(x,z).
\end{split}
\end{equation}
This shows that $h_\lambda\leq \widetilde{h}_\lambda$, and hence also that $h_\lambda=\widetilde{h}_\lambda$, thus completing the proof of~\eqref{eq:yosida_representation}.
\end{proof}

Finally, we show how to apply Proposition~\ref{prop:yosida_representation} to the Hamiltonian~$H$ defined in~\eqref{eqn:H_def} and its Moreau--Yosida regularization~$H_\lambda$ defined by~\eqref{eqn: moreau yosida defn}, in order to deduce that the representation~\eqref{eq:moreau_yosida_representation} holds as detailed in Example~\ref{example: moreau yosida}. 
Consider the case $N=\dsym=\frac{1}{2}d(d+1)$ the dimension of the space~$\Rddsym$ of symmetric $d\times d$ matrices.
Let~$k$ be an indexing function of the set of indices $S_d\coloneqq \{(i,j) \in \{1,\ldots,d\}^2: i\leq j\}$, i.e.\ a bijective mapping from $S_d$ to  $\{1,\ldots, \dsym\}$.
Then consider the linear map~$J$ from~$\Rddsym$ to $\R^{\dsym}$ defined by $(JM)_{k(i,i)}\coloneqq  M_{ii} $ for all $i\in \{1,\ldots, d\}$ and $(JM)_{k(i,j)} = \sqrt{2} M_{ij}$ for all $(i,j)\in S_d$ with $j\neq i$.
It is then straightforward to check that $J$ is a linear isometric isomorphism, where we note that the additional factor of $\sqrt{2}$ appearing above for terms corresponding to off-diagonal entries is used in order to obtain an isometry between the Frobenius norm on~$\Rddsym$ and the Euclidean norm in~$\R^{\dsym}$. 
In particular, we have $(JM)\cdot(J\widetilde{M})=M:\widetilde{M}$ for all $M$, $\widetilde{M}\in \Rddsym$.
It then follows that $H(x,M)=h(x,JM)$ and $H_\lambda(x,M)=h_\lambda(x,JM)$ for all $M\in \Rddsym$ upon choosing $b=-Ja$ and $g=f$.
Using the linearity and isometry of~$J$, it is then easily checked that~\eqref{eq:moreau_yosida_representation} follows from \eqref{eq:yosida_representation}. 

\bibliographystyle{siamplain_NoURL}

\bibliography{references}

\end{document}